\documentclass[11pt]{article}
\usepackage{amsfonts}
\usepackage{amsmath,amssymb}
\usepackage{amsthm}
\usepackage[mathscr]{euscript}
\usepackage{mathrsfs}
\usepackage{indentfirst}
\usepackage{color}
\usepackage{accents}
\usepackage{enumerate}

\newtheorem{theo}{Theorem}[section]
\newtheorem{lem}{Lemma}[section]

\newtheorem{defi}{Definition}[section]
\newtheorem{remark}{Remark}[section]

\numberwithin{equation}{section}

\newcommand{\lbl}[1]{\label{#1}}
\allowdisplaybreaks

\newcommand{\be}{\begin{equation}}
\newcommand{\ee}{\end{equation}}
\newcommand\bes{\begin{eqnarray}} \newcommand\ees{\end{eqnarray}}
\newcommand{\bess}{\begin{eqnarray*}}
\newcommand{\eess}{\end{eqnarray*}}
\newcommand{\bbbb}{\left\{\begin{aligned}}
\newcommand{\nnnn}{\end{aligned}\right.}
\newcommand{\bea}{\begin{align*}}
\newcommand{\eea}{\end{align*}}

\newcommand\ep{\varepsilon}

\newcommand\kk{\left}
\newcommand\rr{\right}
\newcommand\dd{\displaystyle}

\newcommand\dy{{\rm d}y}

\newcommand\yy{\infty}
\newcommand\qq{\eqref}

\newcommand\ud{\underline}

\begin{document}\thispagestyle{empty}
\begin{center}
 {\LARGE\bf Dynamics for a diffusive epidemic model with a free boundary: spreading speed\footnote{This work was supported by NSFC Grants
12171120, 11901541,12301247}}\\[4mm]
{\Large Xueping Li}\\[0.5mm]
{School of Mathematics and Information Science, Zhengzhou University of Light Industry, Zhengzhou, 450002, China}\\[2.5mm]
  {\Large Lei Li}\\[0.5mm]
{College of Science, Henan University of Technology, Zhengzhou, 450001, China}\\[2.5mm]
{\Large Mingxin Wang\footnote{Corresponding author. {\sl E-mail}: mxwang@hpu.edu.cn}}\\[0.5mm]
 {School of Mathematics and Information Science, Henan Polytechnic University, Jiaozuo, 454003, China}
\end{center}

\date{\today}

\begin{quote}
\noindent{\bf Abstract.} We study the spreading speed of a diffusive epidemic model proposed by Li et al. \cite{LL}, where the Stefan boundary condition is imposed at the right boundary, and the left boundary is subject to the homogeneous Dirichlet and Neumann condition, respectively. A spreading-vanishing dichotomy and some sharp criteria were obtained in \cite{LL}. In this paper, when spreading happens, we not only obtain the exact spreading speed of the spreading front described by the right boundary, but derive some sharp estimates on the asymptotical behavior of solution component $(u,v)$. Our arguments depend crucially on some detailed understandings for a corresponding semi-wave problem and a steady state problem.


\end{quote}

\pagestyle{myheadings}
\section{Introduction}
\renewcommand{\thethm}{\Alph{thm}}
{\setlength\arraycolsep{2pt}

The use of reaction-diffusion equations to model the spread of epidemics is a hot topic in biomathematics. Relevant studies not only reveal some interesting propagation phenomena, but also promote the development of corresponding mathematical theories. To study the spread of oral-faecal transmitted epidemics, Hsu and Yang \cite{HY} proposed a reaction-diffusion system
 \bes\label{1.1}
\left\{\!\begin{aligned}
&u_{t}=d_1 u_{xx}-au+H(v), & &t>0,~x\in\mathbb{R},\\
&v_{t}=d_2 v_{xx}-bv+G(u), & &t>0,~x\in\mathbb{R},
\end{aligned}\right.
 \ees
 where $H(v)$ and $G(u)$ satisfy
\begin{enumerate}
\item[{\bf(H)}]\; $H,G\in C^2([0,\yy))$, $H(0)=G(0)=0$, $H'(z),G'(z)>0$ in $[0,\yy)$, $H''(z), G''(z)<0$ in $(0,\yy)$, and $G(H(\hat z)/a)<b\hat{z}$ for some $\hat{z}>0$.
 \end{enumerate}

In this model, $u(t,x)$ and $v(t,x)$ stand for the spatial
concentrations of bacteria and infective human population, respectively, at time $t$ and location $x$ in the one dimensional habitat; $-au$ and $H(v)$ represent the natural death rate of bacterial population and the contribution of infective human to the growth rate of the bacteria, respectively; $-bv$ and $G(u)$ are the fatality rate of infective human population and the infection rates of human population, respectively; $d_1$ and $d_2$, respectively, stand for the diffusion rate of bacteria and infective human. Hsu and Yang in \cite{HY} showed that when
 \bess
  \mathcal{R}_0:=\frac{H'(0)G'(0)}{ab}>1,
  \eess
there exists a $c_*>0$ such that if and only if $c\ge c^*$, \eqref{1.1} has a monotone travelling wave solution $(\phi_1,\phi_2)$ which is unique up to translation and satisfies
 \bes\label{1a.2}
\left\{\!\begin{aligned}
&d_1\phi''_1-c\phi'_1-a\phi_1+H(\phi_2)=0, \;\;x\in\mathbb{R},\\
&d_2\phi''_2-c\phi'_2-b\phi_2+G(\phi_1)=0, \;\;\,x\in\mathbb{R},\\
&\phi_1(-\yy)=\phi_2(-\yy)=0, ~ \phi_1(\yy)=u^*, ~ \phi_2(\yy)=v^*,
\end{aligned}\right.
 \ees
 where $(u^*,v^*)$ is the unique positive root of $au=H(v)$ and $bv=G(u)$.
 The critical value $c^*$ is determined by the characteristic polynomial of the linearized system of \eqref{1.1} at $(0,0)$. More precisely, denote by
 \[P(\lambda, c)=(d_1\lambda^2-c\lambda-a)(d_2\lambda^2-c\lambda-b)-H'(0)G'(0)\]
the characteristic polynomial concerned. The critical value $c^*$ is uniquely given by
 \bes\label{1a.3}
 c^*=\inf\{\bar c>0: {\rm all ~ roots ~ of ~ }P(\lambda, c) ~ {\rm are ~ real ~ for ~ }c\ge\bar{c}\}.
 \ees
Moreover, the dynamics of the corresponding ODE system with positive initial value is govern by $\mathcal{R}_0$. Namely, when $\mathcal{R}_0<1$, $(0,0)$ is globally asymptotically stable; while when $\mathcal{R}_0>1$, then the unique positive equilibrium $(u^*,v^*)$ is globally asymptotically stable.

If $H(v)=cv$, then system \eqref{1.1} reduces to
 \bes\label{1.3}
\left\{\!\begin{aligned}
&u_{t}=d_1 u_{xx}-au+cv, & &t>0,~x\in\mathbb{R},\\
&v_{t}=d_2 v_{xx}-bv+G(u), & &t>0,~x\in\mathbb{R},
\end{aligned}\right.
 \ees
where $G$ satisfies that $G\in C^2([0,\yy))$, $G(0)=0<G'(u)$ in $[0,\yy)$, $G(u)/u$ is strictly decreasing in $(0,\yy)$ and $\lim_{u\to\yy}{G(u)}/{u}<{ab}/{c}$. The corresponding ODE system of \eqref{1.3} was first proposed in \cite{CP} to describe the 1973 cholera epidemic spread in the European Mediterranean regions.

When modeling epidemic, an important issue is to know where the spreading frontier of an epidemic is located, which naturally motivates us to discuss the systems, such as \eqref{1.1}, on the domain whose boundary is unknown and varies over time, instead of the fixed boundary domain or the whole space.
Recently, Li et al. \cite{LL} studied the following epidemic model
  \bes\begin{cases}\label{1a.5}
u_t=d_1u_{xx}-au+H(v), &t>0,~x\in(0,h(t)),\\
v_t=d_2v_{xx}-bv+G(u), &t>0,~x\in(0,h(t)),\\
\mathbb{B}[u](t,0)=\mathbb{B}[v](t,0)=u(t,h(t))=v(t,h(t))=0, \; &t>0,\\
h'(t)=-\mu_1 u_x(t,h(t))-\mu_2v_x(t,h(t)), & t>0,\\
h(0)=h_0, ~ u(0,x)=u_{0}(x), ~ v(0,x)=v_0(x),&0\le x\le h_0,
 \end{cases}
 \ees
 where the initial functions $u_0$ and $v_0$ satisfy
  \begin{enumerate}
\item[{\bf(I)}]\; $w\in C^{2}([0,h_0])$, $w_x(0)>0$, $w(h_0)=w(0)=0<w(x)$ in $(0,h_0)$ when $\mathbb B[w]=w$, $w(h_0)=w_x(0)=0<w(x)$ in $[0,h_0)$ when $\mathbb{B}[w]=w_x$.
 \end{enumerate}
 The operator $\mathbb{B}[w]=w$ or $w_x$, which indicates that the homogeneous Dirichlet or Neumann boundary condition is imposed at the fixed boundary $x=0$, respectively.
They showed that longtime behaviors are governed by a spreading-vanishing dichotomy. That is, one of the following alternatives must happen:

{\it $\underline {Spreading}$:} necessarily $\mathcal{R}_0>1$, $\lim_{t\to\yy}h(t)=\yy$,
\bess\left\{\!\begin{array}{ll}
\dd\lim_{t\to\yy}(u(t,x),v(t,x))=(U(x),V(x)) ~ {\rm in ~ }C_{\rm loc}([0,\yy)) {\rm ~ when ~ operator ~ }\mathbb{B}[w]=w,\\
\dd\lim_{t\to\yy}(u(t,x),v(t,x))=(u^*,v^*) ~ {\rm in ~ }C_{\rm loc}([0,\yy)) {\rm ~ when ~ operator ~ }\mathbb{B}[w]=w_x,
 \end{array}\right.
\eess
where $(U(x),V(x))$ is the unique bounded positive solution of
 \bes\left\{\!\begin{array}{ll}\label{1.2}
-d_1u''=-au+H(v), \;\;&x\in(0,\yy),\\[1mm]
-d_2v''=-bv+G(u), &x\in(0,\yy),\\[1mm]
u(0)=v(0)=0.
 \end{array}\right.
\ees

{\it $\underline {Vanishing}$:} $\lim_{t\to\yy}h(t)<\yy$, $\lim_{t\to\yy}\|u(t,x)+v(t,x)\|_{C([0,h(t)])}=0$.

Moreover, some sharp criteria for spreading and vanishing were also obtained. For example, spreading will happen if $h_0\ge l_0$, where $l_0$ is given by
\bess\left\{\!\begin{array}{ll}
l_0=\dd\pi\sqrt{\frac{ad_2+bd_1+\sqrt{(ad_2-bd_1)^2
 +4d_1d_2H'(0)G'(0)}}{2(H'(0)G'(0)-ab)}}\;\;{\rm ~ when ~ operator ~ }\mathbb{B}[w]=w,\\[4mm]
 l_0=\dd\frac{\pi}{2}\sqrt{\frac{ad_2+bd_1+\sqrt{(ad_2-bd_1)^2
 +4d_1d_2H'(0)G'(0)}}{2(H'(0)G'(0)-ab)}}\;\; {\rm ~ when ~ operator ~  }\mathbb{B}[w]=w_x.
 \end{array}\right.
\eess

The main purpose of this paper is to determine the spreading speed of \eqref{1a.5} when spreading happens. As is seen from the existing literature (see for instance \cite{DL, DGP, DLou, GLZ15, KM15, DWZ17, DLZ, WD}), the spreading speed of free boundary problem is closely related to a corresponding semi-wave problem. For \eqref{1a.5}, its semi-wave problem takes the form of
 \bes\label{1.6}
\left\{\!\begin{aligned}
&d_1\varphi''-c\varphi'-a\varphi+H(\psi)=0, \;\;x>0,\\
&d_2\psi''-c\psi'-b\psi+G(\varphi)=0, \;\;\,x>0,\\
&\varphi(0)=\psi(0)=0, ~ \varphi(\yy)=u^*, ~ \psi(\yy)=v^*.
\end{aligned}\right.
 \ees

We here concern the monotone solution of \eqref{1.6}, i.e., the solution $(\varphi,\psi)$ satisfying $\varphi'(x)>0$ and $\psi'(x)>0$ for $x\ge0$. Our first main result gives a detailed understanding for the monotone solution of \eqref{1.6}.

 \begin{theo}\label{t1.1}Let $\mathcal{R}_0>1$ and $c^*$ be defined by \eqref{1a.3}. Then the following statements are valid.\vspace{-2mm}
 \begin{enumerate}[$(1)$]
   \item Semi-wave problem \eqref{1.6} has a unique monotone solution $(\varphi_c,\psi_c)$ if and only if $c\in[0,c^*)$, which is strictly decreasing in $c\in[0,c^*)$. Moreover, there exist constants $p,q, \alpha>0$ such that $(u^*-\varphi_c(x),v^*-\psi_c(x))={\rm e}^{-\alpha x}(p+o(1),q+o(1))$.\vspace{-2mm}
   \item The map $c\longmapsto(\varphi_c(x),\psi_c(x))$ is continuous from $[0,c^*)$ to $[C^{2}_{\rm loc}([0,\yy))]^2$. Additionally,
    \bess\begin{cases}
\dd\lim_{c\to0}(\varphi_c(x),\psi_c(x))=(U(x),V(x))\;\; ~ {\rm in ~ }\;[C^{2}_{\rm loc}([0,\yy))]^2,\\
\dd\lim_{c\to c^*}(\varphi_c(x),\psi_c(x))=(0,0)\;\; ~ {\rm in ~ }\;[C^{2}_{\rm loc}([0,\yy))]^2,
 \end{cases}
\eess
where $(U,V)$ is the unique bounded positive solution of \eqref{1.2}.

Moreover, we define $\ell_c$ by $\varphi_c(\ell_c)=u^*/2$ or $\psi_c(\ell_c)=v^*/2$. Then $\ell_c\to\yy$, and
   \[\lim_{c\nearrow c^*}(\varphi_c(x+\ell_c),\psi_c(x+\ell_c))=(\phi_1(x),\phi_2(x))~\; ~{\rm in ~ }\;[C^{2}_{\rm loc}(\mathbb{R})]^2,\]
 where $(\phi_1,\phi_2)$ is the travelling wave solution of \eqref{1a.2} with speed  $c^*$.\vspace{-2mm}
\item For any $\mu_1>0$ and $\mu_2>0$, there exists a unique $c_{\mu_1,\mu_2}\in (0,c^*)$ such that
  \bes
  \mu_1\varphi'_{c_{\mu_1,\mu_2}}(0)+\mu_2\psi'_{c_{\mu_1,\mu_2}}(0)
  =c_{\mu_1,\mu_2}.
  \lbl{1.7}\ees
Moreover, $c_{\mu_1,\mu_2}\to c^*$ as one of $\mu_1$ and $\mu_2$ tends to infinity, and $c_{\mu_1,\mu_2}\to0$ as $\mu_1+\mu_2\to0$.
 \end{enumerate}
 \end{theo}

With the help of above results as well as some suitable upper and lower solutions, we can obtain our conclusion on the spreading speed of \eqref{1a.5} when spreading occurs.

 \begin{theo}\label{t1.2}Let $\mathcal{R}_0>1$ and spreading happen for \eqref{1a.5}. Then the unique solution $(u,v,h)$ of \eqref{1a.5} satisfies
 \bess\begin{cases}
\dd\lim_{t\to\yy}\frac{h(t)}{t}=c_{\mu_1,\mu_2},\\
\dd\lim_{t\to \yy}\max_{x\in[0,ct]}\big(|u(t,x)-U(x)|+|v(t,x)-V(x)|\big)=0, ~ \forall\, c\in[0,c_{\mu_1,\mu_2}) ~ {\rm when ~} \mathbb{B}[w]=w,\\
\dd\lim_{t\to \yy}\max_{x\in[0,ct]}\big(|u(t,x)-u^*|+|v(t,x)-v^*|\big)=0, ~ \forall\, c\in[0,c_{\mu_1,\mu_2}) ~ {\rm when ~} \mathbb{B}[w]=w_x,
 \end{cases}
\eess
where $(u^*,v^*)$ is the unique positive root of $au=H(v)$ and $bv=G(u)$, $(U,V)$ is the unique bounded positive solution of \eqref{1.2}, and $c_{\mu_1,\mu_2}$ is uniquely determined by \eqref{1.7}.
 \end{theo}

To be more readable, here we add something about the spreading speeds of \eqref{1a.5} and its corresponding Cauchy problem \eqref{1.1}. As is well known to us (see for example \cite{TZ}), a number $\bar c>0$ is said to be the asymptotical spreading speed (or spreading speed) for a nonnegative function $u$ which is usually a unique solution of an equation or system, if $u$ satisfies
 \bess\begin{cases}
 \dd\liminf_{t\to\yy}\inf_{|x|\le ct}u(t,x)>0, ~ ~ \forall\, c\in(0,\bar{c}),\\
 \dd\lim_{t\to\yy}\sup_{|x|\ge ct}u(t,x)=0, ~ ~ \forall\, c>\bar{c}.
 \end{cases}
\eess
According to Theorem \ref{t1.2}, when spreading happens for \eqref{1a.5}, $c_{\mu_1,\mu_2}$ is the asymptotical spreading speed of solution component $(u,v)$ of \eqref{1a.5}, which, for convenience, is often called the spreading speed for problem \eqref{1a.5}. For the Cauchy problem \eqref{1.1}, it is not hard to show that $c^*$ is the asymptotical spreading speed for \eqref{1.1} with nonnegative and compactly supported initial functions, namely, its solution $(u,v)$ satisfying
\bess\begin{cases}
 \dd\lim_{t\to\yy}\max_{|x|\le ct}\big(|u(t,x)-u^*|+|v(t,x)-v^*|\big)=0, ~ ~ \forall\, c\in(0,c^*),\\
 \dd\lim_{t\to\yy}\max_{|x|\ge ct}\big(u(t,x)+v(t,x)\big)=0, ~ ~ \forall\, c>c^*.
 \end{cases}
\eess
In view of Theorem \ref{t1.2}, the asymptotical spreading speed $c_{\mu_1,\mu_2}$ of \eqref{1a.5} converges to $c^*$, the asymptotical spreading speed of the Cauchy problem \eqref{1.1} as $\mu_1$ or $\mu_2$ tends to infinity.  On the other hand, it is easy to see that problem \eqref{1a.5} with double free boundaries $g(t)$ and $h(t)$ shares the same dynamics with \eqref{1a.5} with $\mathbb{B}[w]=w_x$. Moreover, as in \cite{DuG}, the Cauchy problem \eqref{1.1} can be seen as the limiting problem of \eqref{1a.5} with double free boundaries as $\mu_1$ or $\mu_2$ tends to infinity.

In a recent work, Wang et al \cite{WND1} obtained some sharp estimates on the asymptotical behaviors of the solution to the West-Nile virus model with local diffusions and free boundaries. It is expected that similar results hold true for \eqref{1a.5}. However, when operator $\mathbb{B}[w]=w$, some new difficulties and results maybe appear since the solution component $(u,v)$ converges to a non-constant steady state solution when spreading happens. This issue will be studied in a separate work.

When the operator $\mathbb{B}[w]=w_x$, problem \qq{1a.5} is equivalent to the free boundary problem with double free boundaries, which was studied by Wang and Du \cite{WD} for case $H(v)=cv$. When the operator $\mathbb{B}[w]=w$, the problem \qq{1a.5} is different from the case where both sides are free boundaries. There have been lots of works concerning the spreading speed of the model where the left side is fixed and subject to homogeneous mixed boundary condition. For instance, please see \cite{ZW18} for reaction-diffusion-advection equation, \cite{Wjfa16} for the logistic equation with sign-changing coefficient and time-periodic environment, \cite{GW15, Wu15, WZna17, LHW19} for the competition model, \cite{WZZ21} for the competition model with seasonal succession, and  \cite{WZjde18} for the prey-predator model.

For the work on spreading speeds of nonlocal diffusion equations or systems where only one side is a free boundary, see for example \cite{LWZjde22, LLW-Z22,LiWang}. For other epidemic models with free boundary, one can see \cite{LZ,WND} for the West-Nile virus model with local diffusions, \cite{DN2,DN1} for the West-Nile virus model with nonlocal diffusions, and \cite{HW1,CLWY} for SIR and SIRS epidemic models, respectively. Besides, the interested readers can refer to \cite{DMZ,DDL,KMY} or the survey paper \cite{Du22} and the references therein for more recent progress on the free boundary problem arising from ecology.

This paper is arranged as follows. Section 2 involves the proof of Theorem \ref{t1.1}. We first obtain the existence of the monotone solution of \eqref{1.6} by using the upper and lower solutions method; then the uniqueness and other properties in Theorem \ref{t1.1} are derived by virtue of some basic analysis and the standard theory of elliptic equations, such as the strong maximum principle and the Hopf lemma. Section 3 is devoted to the proof of Theorem \ref{t1.2}. Taking advantage of the monotone solutions of some perturbed problems of \eqref{1.6}, we construct some suitable upper and lower solutions which help us to derive the results as wanted.

Throughout this paper, we always assume that $\mathcal{R}_0>1$, $(u,v,h)$ is the unique solution of \eqref{1a.5} and spreading happens for \eqref{1a.5}.

\section{The semi-wave problem \eqref{1.6}}{\setlength\arraycolsep{2pt}

In this section, we will prove Theorem \ref{t1.1} by following the similar lines as in \cite{WD} or \cite{WND}. To this end, we first give the definitions of upper and lower solutions for \eqref{1.6}.

\begin{defi}\label{d2.1}Suppose that $(\bar{\varphi},\bar{\psi})$ and $(\ud \varphi,\ud \psi)$ are continuous in $[0,\yy)$ and twice continuously differentiable in $(0,\yy)$ but except for some finite points. Moreover, the following assumptions hold.
\begin{enumerate}[$(1)$]
\item
 \bes
\left\{\!\begin{aligned}
&d_1\bar\varphi''-c\bar\varphi'-a\bar\varphi+H(\bar\psi)\le0, \;\;0<x<\yy,~ x\notin\{x_1,x_2,\cdots, x_n\},\\
&d_2\bar\psi''-c\bar\psi'-b\bar\psi+G(\bar\varphi)\le0, \;\;\,0<x<\yy,~ x\notin\{x_1,x_2,\cdots, x_n\},\\
&\bar\varphi(0)=\bar\psi(0)=0, ~ \bar\varphi(\yy)=u^*, ~ \bar\psi(\yy)=v^*.
\end{aligned}\right.
 \lbl{2b.1}\ees
\item
 \bes
\left\{\!\begin{aligned}
&d_1\ud\varphi''-c\ud\varphi'-a\ud\varphi+H(\ud\psi)\ge0,\;\;0<x<\yy,~ x\notin\{\tilde x_1,\tilde x_2,\cdots, \tilde x_m\},\\
&d_2\ud\psi''-c\ud\psi'-b\ud\psi+G(\ud\varphi)\ge0, \;\;\,0<x<\yy,~ x\notin\{\tilde x_1,\tilde x_2,\cdots, \tilde x_m\},\\
&\ud\varphi(0)=\ud\psi(0)=0, ~ \ud\varphi(\yy)\le u^*, ~ \ud\psi(\yy)\le v^*.
\end{aligned}\right.\lbl{2b.2}
 \ees
\item
  \bes
\left\{\!\begin{aligned}
&\bar{\varphi}'(x^+_i)\le\bar{\varphi}'(x^-_i), ~\bar{\psi}'(x^+_i)\le\bar{\psi}'(x^-_i), & &i=\{1,2,\cdots, n\},\\
&\ud{\varphi}'(\tilde x^+_i)\ge\ud{\varphi}'(\tilde x^-_i), ~\ud{\psi}'(\tilde x^+_i)\ge\ud{\psi}'(\tilde x^-_i), & &i=\{1,2,\cdots, m\}.
\end{aligned}\right.\lbl{2b.3}
 \ees
 \end{enumerate}
Then we call that $(\bar{\varphi},\bar{\psi})$ and $(\ud \varphi,\ud \psi)$ are the upper and lower solutions of \eqref{1.6}, respectively.
\end{defi}

To construct the suitable  upper and lower solutions, we need some properties of the characteristic polynomial $P(\lambda, c)$, which can be seen in \cite[Lemma 2.1]{HY} or  \cite[Lemma 3.3]{WND}. For convenience, we list it below. Recall
 \[P(\lambda, c)=(d_1\lambda^2-c\lambda-a)(d_2\lambda^2-c\lambda-b)-H'(0)G'(0).\]

 \begin{lem}\label{l2.1}Let $c^*$ be defined by \eqref{1a.3}. Then the following statements are valid.\vspace{-2mm}
 \begin{enumerate}[$(1)$]
   \item For any $c>c^*$, $P(\lambda, c)$ has four different real roots $\lambda_i(c)$ with $i=1,2,3,4$, which depend continuously on $c>c^*$, and
   \bess
   \lambda_1(c)<\lambda^-_j<0<\lambda_2(c)<\lambda_3(c)
   <\lambda^+_j<\lambda_4(c),\;\;j=1,2,
   \eess
 where
   \[\lambda^{\pm}_1=\frac{c\pm\sqrt{c^2+4d_1a}}{2d_1}, ~ ~ \lambda^{\pm}_2=\frac{c\pm\sqrt{c^2+4d_2b}}{2d_2}.\]
 Moreover, $P(\lambda, c)>0$ in $(-\yy,\lambda_1(c))\cup(\lambda_2(c),\lambda_3(c))\cup(\lambda_4(c),\yy)$, and $P(\lambda, c)<0$ in $(\lambda_1(c),\lambda_2(c))\cup(\lambda_3(c),\lambda_4(c))$.\vspace{-2mm}
   \item For $c=c^*$, the statement $(1)$  still holds but with $\lambda_2(c)=\lambda_3(c)$.\vspace{-2mm}
   \item For any $c\in[0,c^*)$, $P(\lambda, c)$ has two different real roots $\lambda_1(c)$ and $\lambda_4(c)$, which are continuous in $c\ge0$, and also has a pair of conjugate  complex roots.\vspace{-2mm}
 \end{enumerate}
 \end{lem}

Now we construct the suitable  upper and lower solutions, and then use \cite[Proposition 2.6]{WND} to show the existence of monotone solution of \eqref{1.6}, which will be done by several lemmas.

 \begin{lem}\label{l2.2}Define
 \begin{align*}
\dd\bar\varphi(x)=\left\{\!\begin{aligned}
&u^*\sin(kx), & &0\le x\le\frac{\pi}{2k},\\
&u^*, & &x>\frac{\pi}{2k},
\end{aligned}\right.
~ ~ ~ \dd\bar\psi(x)=\left\{\!\begin{aligned}
&v^*\sin(kx), & &0\le x\le\frac{\pi}{2k},\\
&v^*, & &x>\frac{\pi}{2k}.
\end{aligned}\right.
  \end{align*}
Then there exists $K>0$ such that for all $c\ge0$, $(\bar\varphi,\bar{\psi})$ is an  upper solution of \eqref{1.6} when $k>K$.
 \end{lem}

\begin{proof}Clearly, $(\bar{\varphi},\bar\psi)\in C^1([0,\yy))\cap C^2([0,\yy)\setminus\{\frac{\pi}{2k}\})$ and the assumption \qq{2b.3} of Definition \ref{d2.1} holds. It thus remains to show that the assumption \qq{2b.1} is valid. It is easy to see that for $x>\frac{\pi}{2k}$,
 \bess
\left\{\!\begin{aligned}
&d_1\bar\varphi''-c\bar\varphi'-a\bar\varphi+H(\bar\psi)=0, \\
&d_2\bar\psi''-c\bar\psi'-b\bar\psi+G(\bar\varphi)=0.
\end{aligned}\right.
 \eess
 For $x\in(0,\frac{\pi}{2k})$, direct computations yield
 \bess
 &d_1\bar\varphi''-c\bar\varphi'-a\bar\varphi+H(\bar\psi)\le(-d_1k^2u^*+H'(0)v^*)\sin (kx)\le0,\\
  &d_2\bar\psi''-c\bar\psi'-b\bar\psi+G(\bar\varphi)\le(-d_2k^2v^*+G'(0)u^*)\sin (kx)\le0
 \eess
provided that $k$ is sufficiently large. The proof is finished.
 \end{proof}

Now we focus on the construction of a suitable  lower solution to \eqref{1.6} for $c\in[0,c^*)$ . By perturbing the linearized system of \eqref{1.6} at $(0,0)$, we obtain the following system
 \bes\label{2.1}
\left\{\!\begin{aligned}
&d_1\vartheta''_1-c\vartheta'_1-a\vartheta_1+(H'(0)-\ep)\vartheta_2=0,\\
&d_2\vartheta''_2-c\vartheta'_2-b\vartheta_2+(G'(0)-\ep)\vartheta_1=0,
\end{aligned}\right.
 \ees
where $0<\ep\ll1$. Since $c\in[0,c^*)$, by Lemma \ref{l2.1} we see that the characteristic polynomial $P(\lambda, c,\ep)$ of \eqref{2.1} has a pair of conjugate complex roots if $\ep$ is small enough. Denote by $\lambda_{\pm}=\alpha\pm i\beta$ the conjugate complex roots of $P(\lambda, c,\ep)$. Then \eqref{2.1} has a complex valued solution $(\varphi,\psi)$:
 \[(\varphi,\psi)=\kk(\frac{b+c\lambda_+
 -d_2\lambda^2_+}{G'(0)-\ep},\; 1\rr){\rm e}^{\lambda_+x}.\]
 Let
 \[A={\rm Re}[b+\lambda_+c-d_2\lambda^2_+] ~ ~ {\rm and ~ ~ }B={\rm Im}[b+\lambda_+c-d_2\lambda^2_+].\]
Clearly, $A^2+B^2\neq0$. Set  $\omega=\arctan\frac{B}{A}\in[-\frac{\pi}{2},\frac{\pi}{2}]$.
Simple computations yield
 \bess
 ({\rm Im}[\varphi],{\rm Im}[\psi])=\kk(\frac{\sqrt{A^2+B^2}}{G'(0)-\ep}\sin(\beta x+\omega){\rm e}^{\alpha x},\;\sin(\beta x){\rm e}^{\alpha x}\rr),
 \eess
 which is a real valued solution of \eqref{2.1}.
  Define
  \begin{align*}
&\dd\ud\varphi(x)=\left\{\!\begin{aligned}
&\sigma{\rm Im}[\varphi], & &\frac{2\pi-\omega}{\beta}< x<\frac{3\pi-\omega}{\beta},\\
&0, & &{\rm otherwise},
\end{aligned}\right.~ ~
\dd\ud\psi(x)=\left\{\!\begin{aligned}
&\sigma{\rm Im}[\psi], & &\frac{2\pi}{\beta}< x<\frac{3\pi}{\beta},\\
&0, & &{\rm otherwise},
\end{aligned}\right.
  \end{align*}
where $\sigma>0$ is determined later.

\begin{lem}\label{l2.3}For $c\in[0,c^*)$, the pair $(\ud\varphi,\ud\psi)$ defined as above is a lower solution of \eqref{1.6} provided that $\sigma>0$ is small enough.
  \end{lem}

\begin{proof}Obviously, $(\ud\varphi,\ud\psi)$ is continuous in $x\ge0$, and
  \bess
  \ud\varphi'\kk(x_1^+\rr)\ge\ud\varphi'
  \kk(x_1^-\rr), ~ ~ \ud\varphi'\kk(x_2^+\rr)\ge\ud\varphi'
  \kk(x_2^-\rr), \;\;\;
  \ud\psi'\kk(x_3^+\rr)\ge\ud\psi'\kk(x_3^-\rr), ~ ~ \ud\psi'\kk(x_4^+\rr)\ge\ud\psi'\kk(x_4^-\rr),
    \eess
where $x_1=\frac{2\pi-\omega}{\beta}$, $x_2=\frac{3\pi-\omega}{\beta}$, $x_3=\frac{2\pi}{\beta}$ and $x_4=\frac{3\pi}{\beta}$. Therefore, \qq{2b.3} holds. It remains to verify \qq{2b.2}. Suppose $\omega\in[0,\frac{\pi}{2}]$ since the other case can be handled similarly.

If $x\in(\frac{2\pi-\omega}{\beta},\frac{2\pi}{\beta})$, then $\ud\varphi=\sigma{\rm Im}[\varphi]$ and $\ud\psi=0$. A straightforward calculation shows
  \bess
  d_1\ud\varphi''-c\ud\varphi'-a\ud\varphi+H(\ud\psi)&=&\sigma(\ep-H'(0))\sin(\beta x){\rm e}^{\alpha x}>0,\\
  d_2\ud\psi''-c\ud\psi'-b\ud\psi+G(\ud\varphi)&=&G(\ud\varphi)>0.
  \eess

If $x\in(\frac{2\pi}{\beta},\frac{3\pi-\omega}{\beta})$, then $\ud\varphi=\sigma{\rm Im}[\varphi]$ and $\ud\psi=\sigma{\rm Im}[\psi]$. It follows that
  \bess
  &d_1\ud\varphi''-c\ud\varphi'-a\ud\varphi+H(\ud\psi)=(\ep+o(1))\sigma {\rm Im}[\psi]>0,\\
  &d_2\ud\psi''-c\ud\psi'-b\ud\psi+G(\ud\varphi)=(\ep+o(1))\sigma {\rm Im}[\varphi]>0,
  \eess
where $o(1)\to0$ uniformly for $x\in(\frac{2\pi}{\beta},\frac{3\pi-\omega}{\beta})$ as $\sigma\to0$.

If $x\in(\frac{3\pi-\omega}{\beta},\frac{3\pi}{\beta})$, then $\ud\varphi=0$ and $\ud\psi=\sigma{\rm Im}[\psi]$. Thus
  \bess
  d_1\ud\varphi''-c\ud\varphi'-a\ud\varphi+H(\ud\psi)&=&H(\ud\psi)>0,\\
  d_2\ud\psi''-c\ud\psi'-b\ud\psi+G(\ud\varphi)&=&(\ep-G'(0))\sigma {\rm Im}[\tilde\varphi]>0.
  \eess

If $x\notin(\frac{2\pi-\omega}{\beta},\frac{3\pi}{\beta})$, then $(\ud\varphi,\ud\psi)$ is identical to $(0,0)$ and the desired inequalities hold.   \end{proof}

Now we are in the position to prove the existence and uniqueness of monotone solution of \eqref{1.6}.

\begin{lem}\label{l2.4}Problem \eqref{1.6} has a unique monotone solution if and only if $c\in[0,c^*)$.
\end{lem}
\begin{proof}The proof is divided into two steps.

{\bf Step 1.} In this step, we show there is no monotone solution of \eqref{1.6} if $c\ge c^*$. By way of contradiction, we assume that $(\varphi,\psi)$ is a monotone solution of \eqref{1.6} with $c\ge c^*$. Notice that $\varphi$ and $\psi$ are strictly increasing in $x>0$ and $(\varphi, \psi)$ convergence to $(u^*,v^*)$ as $x\to\yy$. We have $\liminf_{x\to\yy}\varphi'(x)=\liminf_{x\to\yy}\psi'(x)=0$. If $\limsup_{x\to\yy}\varphi'(x)>0$, by the Fluctuation Lemma (\cite[Lemma 2.2]{WZou}), there exists a sequence $\{x_n\}$ with $\lim_{n\to\yy}x_n=\yy$ and $\varphi''(x_n)=0$ such that $\lim_{n\to\yy}\varphi'(x_n)=\limsup_{x\to\yy}\varphi'(x)>0$. Substituting such $x_n$ into the first equation of \eqref{1.6} and letting $n\to\yy$, we derive
\[0>-c\lim_{n\to\yy}\varphi'(x_n)=au^*-H(v^*)=0,\]
which implies $\limsup_{x\to\yy}\varphi'(x)=0$. Similarly, $\limsup_{x\to\yy}\psi'(x)=0$. Thus $\lim_{x\to\yy}\varphi'(x)=\lim_{x\to\yy}\psi'(x)=0$. It then immediately follows from \eqref{1.6} that $\lim_{x\to\yy}\varphi''(x)=\lim_{x\to\yy}\psi''(x)=0$. Thus we can multiply ${\rm e}^{-\lambda_2(c)x}$ (defined in Lemma \ref{l2.1}) to the equations of \eqref{1.6} and integrate them from $0$ to $\yy$ which leads to
 \bes\begin{cases}
\dd(d_1\lambda^2_2(c)-c\lambda_2(c)-a)\int_{0}^{\yy}\varphi(x){\rm e}^{-\lambda_2(c)x}{\rm d}x+H'(0)\int_{0}^{\yy}\psi(x){\rm e}^{-\lambda_2(c)x}{\rm d}x>0,\\[3mm]
\dd(d_2\lambda^2_2(c)-c\lambda_2(c)-b)\int_{0}^{\yy}\psi(x){\rm e}^{-\lambda_2(c)x}{\rm d}x+G'(0)\int_{0}^{\yy}\varphi(x){\rm e}^{-\lambda_2(c)x}{\rm d}x>0.
 \end{cases}\lbl{2.2a}\ees
However, as $\lambda^-_j<\lambda_2(c)<\lambda^+_j$ for $j=1,2$, by Lemma \ref{l2.1} we have
 \bess
 &d_1\lambda^2_2(c)-c\lambda_2(c)-a<0,\;\;\;d_2\lambda^2_2(c)-c\lambda_2(c)-b<0,\\
 &(d_1\lambda^2_2(c)-c\lambda_2(c)-a)(d_2\lambda^2_2(c)-c\lambda_2(c)-b)-H'(0)G'(0)=0,
 \eess
which implies the matrix
 \bess
 \kk(\begin{array}{cc}
 \frac{d_1\lambda^2_2(c)-c\lambda_2(c)-a}{H'(0)}\;\;&1\\
 1\;\;&\frac{d_2\lambda^2_2(c)-c\lambda_2(c)-b}{G'(0)}\end{array}\rr)\eess
is semi-negative definite. So \qq{2.2a} is impossible.  This completes the step 1.

{\bf Step 2.} In this step, we prove \eqref{1.6} has a unique monotone solution if $c\in[0,c^*)$. It is not hard to verify that the  upper and lower solutions constructed in Lemmas \ref{l2.2} and \ref{l2.3} satisfy
\bess
 \sup_{0\le x\le y}\ud\varphi(x)\le\bar{\varphi}(y), ~ \sup_{0\le x\le y}\ud\psi(x)\le\bar{\psi}(y), ~ ~ \forall\, ~ y\ge0
 \eess
provided that $k$ is large enough and $\sigma$ is suitably small. Hence the existence can  be directly derived by \cite[Proposition 2.6]{WND}. It thus remains to show the uniqueness.

Let $(\varphi_1,\psi_1)$ be another monotone solution of \eqref{1.6}. Thanks to the Hopf lemma, we have $\varphi'(0)>0$, $\psi'(0)>0$, $\varphi'_1(0)>0$ and $\psi'_1(0)>0$. Recall $(\varphi(\yy),\psi(\yy))=(\varphi_1(\yy),\psi_1(\yy))=(u^*,v^*)$. Thus there exists $\rho>1$ such that $(\rho\varphi(x),\rho\psi(x))\ge(\varphi_1(x),\psi_1(x))$ in $[0,\yy)$. Define
\[\rho^*=\inf\{\rho>1: (\rho\varphi(x),\rho\psi(x))\ge(\varphi_1(x),\psi_1(x)) ~ ~ \forall\, x\in(0,\yy)\}.\]
Clearly, $\rho^*\ge1$ and $(\rho^*\varphi(x),\rho^*\psi(x))\ge(\varphi_1(x),\psi_1(x))$ for $x\in[0,\yy)$. We now show $\rho^*=1$. Argue on the contrary that $\rho^*>1$. By \eqref{1.6}, we see
\[-d_1(\rho^*\varphi-\varphi_1)''+c(\rho^*\varphi-\varphi_1)'+a(\rho^*\varphi-\varphi_1)>0 ~ ~{\rm in ~ }(0,\yy).\]
Moreover, $(\rho^*\varphi-\varphi_1)(0)=0$. By the strong maximum principle and the Hopf lemma, $(\rho^*\varphi-\varphi_1)(x)>0$ in $(0,\yy)$ and $(\rho^*\varphi-\varphi_1)'(0)>0$. Note that $(\rho^*\varphi-\varphi_1)(\yy)=(\rho^*-1)u^*>0$. There exists a small $\ep>0$ such that $(\rho^*-\ep)\varphi(x)\ge\varphi_1(x)$ in $[0,\yy)$ and $\rho^*-\ep>1$. Analogously, $(\rho^*-\ep_1)\psi(x)\ge\psi_1(x)$ in $[0,\yy)$ and $\rho^*-\ep_1>1$ for some small $\ep_1>0$. This obviously contradicts the definition of $\rho^*$. So $\rho^*=1$, and $(\varphi(x),\psi(x))\ge(\varphi_1(x),\psi_1(x))$ for $x\ge0$. Exchanging the positions of $(\varphi,\psi)$ and $(\varphi_1,\psi_1)$, we can deduce $(\varphi(x),\psi(x))\le(\varphi_1(x),\psi_1(x))$ for $x\ge0$. Therefore, the uniqueness is obtained. The proof is finished.
\end{proof}

According to the above lemma, we know that \eqref{1.6} has a unique monotone solution $(\varphi,\psi)$ if and only if $c\in[0,c^*)$. Notice that $(\varphi(\yy),\psi(\yy))=(u^*,v^*)$. One naturally wonders what the rate of this convergence is. In the next lemma we will show that this convergence is exponential.

\begin{lem}\label{l2.5}Suppose $c\in[0,c^*)$ and $(\varphi,\psi)$ is the unique monotone solution of \eqref{1.6}. Then there exist constants $p,q, \alpha>0$ such that
 \bes
 (u^*-\varphi(x),v^*-\psi(x))={\rm e}^{-\alpha x}(p+o(1),q+o(1)).
 \lbl{2.3a}\ees
\end{lem}
\begin{proof}Firstly, the linearized system of \eqref{1.6} at $(u^*,v^*)$ takes the from of
\bes\label{2.2}
\left\{\!\begin{aligned}
&d_1\varphi''-c\varphi'-a\varphi+H'(v^*)\psi=0,\\
&d_2\psi''-c\psi'-b\psi+G'(u^*)\varphi=0.
\end{aligned}\right.
 \ees
 If $(p,q){\rm e}^{\lambda x}$ is a nontrivial solution of \eqref{2.2}, then there must hold:
 \[A(\lambda)(p,q)^T=(0,0)^T ~ ~ {\rm and } ~ ~ \hat{P}(\lambda)=0,\]
 where
 \bess
 &A(\lambda)=\begin{pmatrix}
  d_1\lambda^2-c\lambda-a & H'(v^*) \\
  G'(u^*) & d_2\lambda^2-c\lambda-b
  \end{pmatrix},\\[2mm]
 &\hat{P}(\lambda)={\rm det}A(\lambda)=(d_1\lambda^2-c\lambda-a)
 (d_2\lambda^2-c\lambda-b)-H'(v^*)G'(u^*).
 \eess
It is easy to see that $\hat{P}(\lambda^{\pm}_i)<0$ for $i=1,2$, $\hat{P}(0)=ab-H'(v^*)G'(u^*)>0$ and $\hat{P}(\pm\yy)=\yy$ where $\lambda^{\pm}_i$ are defined in Lemma \ref{l2.1}. Thus there exist four distinct real roots $\hat{\lambda}_i$ for $\hat{P}(\lambda)=0$ with
 \[ \hat{\lambda}_1<\lambda^-_j<\hat{\lambda}_2<0<\hat{\lambda}_3
 <\lambda^+_j<\hat{\lambda}_4,\;\;j=1,2,\]
which implies that the first order ODE system satisfied by $(\varphi,\varphi',\psi,\psi')$ has a critical point $(u^*,0,v^*,0)$ that is a saddle point. Then by the standard stable manifold theory (see
\cite[Theorem 4.1 of Ch. 13]{CLev} and its proof, pp330, or the stable manifold theorem and its proof of \cite{Law}, pp107), we deduce that $(\varphi,\psi)\to(u^*,v^*)$ exponentially as $x\to\yy$. Let $(\hat{\varphi},\hat{\psi})=(u^*-\varphi,v^*-\psi)$. Then $(\hat{\varphi},\hat{\psi})$ satisfies
 \bes\label{2.3}
\left\{\!\begin{aligned}
&d_1\hat\varphi''-c\hat\varphi'-a\hat\varphi+H'(v^*)\hat\psi
+\ep_1(x)\hat\varphi=0,\\[1mm]
&d_2\hat\psi''-c\hat\psi'-b\hat\psi+G'(u^*)\hat\varphi+\ep_2(x)\hat\psi=0,
\end{aligned}\right.
 \ees
where
 \bess
 &\ep_1(x)=\dd\frac{au^*-H(\psi)-H'(v^*)(v^*-\psi)}{\varphi}\to0 ~ ~ {\rm exponentially ~ as ~ }x\to\yy,\\[1mm]
 &\ep_2(x)=\dd\frac{bv^*-G(\varphi)-G'(u^*)(u^*-\varphi)}{\psi}\to0~ ~ {\rm exponentially ~ as ~ }x\to\yy.
 \eess
Recall that the characteristic equation $\hat{P}(\lambda)=0$ of \eqref{2.2} has four different real roots. Hence \eqref{2.2} has four linearly independent solutions $\Phi_i=(p_i,q_i){\rm e}^{\hat{\lambda}_ix}$ with $A(\hat{\lambda}_i)(p_i,q_i)^T=(0,0)^T$. In view of \cite[Theorem 8.1 of Ch.  3, pp92]{CLev}, we know that \eqref{2.3} has four linearly independent solutions $\hat{\Phi}_i$ satisfying $\hat\Phi_i(x)=(1+o(1))\Phi_i(x)$ as $x\to\yy$, $i=1,2,3,4$. So the solution $(\hat{\varphi},\hat{\psi})$ of \eqref{2.3} can be represented by
 \[(\hat{\varphi},\hat{\psi})=\sum_{i=1}^{4}a_i\hat{\Phi}_i.\]
Noticing that $(\hat{\varphi}(\yy),\hat{\psi}(\yy))=(0,0)$ and $\hat{\lambda}_4>\hat{\lambda}_3>0$, we immediately derive $a_3=a_4=0$. We now show $a_2\neq0$. Otherwise, we have $a_1\neq0$.  Namely,
 \[(\hat{\varphi},\hat{\psi})=a_1\hat{\Phi}_1=(1+o(1))a_1(p_1,q_1){\rm e}^{\hat{\lambda}_1x} ~ ~ {\rm as ~ }x\to\yy.\]
Since the four elements of $A(\hat{\lambda}_1)$ are positive, we know $p_1q_1<0$ which implies that one of $\hat{\varphi}(x)$ and $\hat{\psi}(x)$ is negative for $x\gg1$. This contradicts $\hat{\varphi}(x)>0$ and $\hat{\psi}(x)>0$ for $x>0$. So $a_2\neq0$. It is easy to verify that $p_2q_2>0$. Choose $p_2>0$ and $q_2>0$.  Due to $\hat{\lambda}_1<\hat{\lambda}_2<0$, we see $(\hat{\varphi},\hat{\psi})=(1+o(1))a_2(p_2,q_2){\rm e}^{\hat{\lambda}_2x}$ as $x\to\yy$. By the positivity of $\hat{\varphi}$ and $\hat{\psi}$, we have $a_2>0$. Noticing that $-\hat{\lambda}_2>0$, we obtain \qq{2.3a}.
\end{proof}

Now we discuss the asymptotical behaviors of the monotone solution $(\varphi,\psi)$ of \eqref{1.6} for $c\in[0,c^*)$. To stress the dependence, we rewrite $(\varphi,\psi)$ as $(\varphi_c,\psi_c)$.

\begin{lem}\label{l2.6}Let $c\in[0,c^*)$ and $(\varphi_c,\psi_c)$ be defined as above. Then the following statements hold.\vspace{-2mm}
\begin{enumerate}[$(1)$]
 \item  The unique monotone solution $(\varphi_c,\psi_c)$ is strictly decreasing in $c\in[0,c^*)$, i.e., $0\le c_1<c_2<c^*$ implies $\varphi_{c_1}(x)>\varphi_{c_2}(x)$ and $\psi_{c_1}(x)>\psi_{c_2}(x)$ for all $x>0$.\vspace{-2mm}
 \item The map $c\longmapsto(\varphi_c(x),\psi_c(x))$ is continuous from $[0,c^*)$ to $[C^{2}_{\rm loc}([0,\yy))]^2$. Additionally,
  \bes\begin{cases}\label{2.4}
\dd\lim_{c\to0}(\varphi_c,\psi_c)=(U,V) ~ \;\;{\rm in ~ }\;[C^{2}_{\rm loc}([0,\yy))]^2,\\
\dd\lim_{c\to c^*}(\varphi_c,\psi_c)=(0,0) ~\;\; {\rm in ~ }\;[C^{2}_{\rm loc}([0,\yy))]^2,
 \end{cases}
\ees
where $(U,V)$ is the unique bounded positive solution of \eqref{1.2}.

Moreover, we define $\ell_c$ by $\varphi_c(\ell_c)=u^*/2$ or $\psi_c(\ell_c)=v^*/2$. Then $\lim_{c\to c^*}\ell_c=\yy$, and
\[\lim_{c\to c^*}(\varphi_c(x+\ell_c),\psi_c(x+\ell_c))=
(\phi_1(x),\phi_2(x))~ ~ {\rm  in ~ }[C^{2}_{\rm loc}(\mathbb{R})]^2,\]
where $(\phi_1(x),\phi_2(x))$ is the travelling wave solution of \eqref{1a.2} with speed  $c^*$.\vspace{-2mm}
\item For any positive constants $\mu_1$ and $\mu_2$, there exists a unique $c_{\mu_1,\mu_2}\in (0,c^*)$ such that
\[\mu_1\varphi'_{c_{\mu_1,\mu_2}}(0)+\mu_2\psi'_{c_{\mu_1,\mu_2}}(0)=c_{\mu_1,\mu_2}.\]
Moreover, $c_{\mu_1,\mu_2}\to c^*$ as one of $\mu_1$ and $\mu_2$ tends to infinity, and $c_{\mu_1,\mu_2}\to0$ as $\mu_1+\mu_2\to0$.\vspace{-2mm}
\end{enumerate}
\end{lem}
\begin{proof}(1) This assertion can be proved by the similar arguments as in \cite[Theorem 4.6]{WD}. Thus the details are omitted here.

(2) Choose an arbitrary $\tilde c\in[0,c^*)$ and a sequence $\{c_n\}\subseteq[0,c^*)$. We first assume that $c_n$ decreases to $\tilde{c}$ as $n\to\yy$. Denote by $(\tilde\varphi,\tilde\psi)$ and $(\varphi_n,\psi_n)$ the corresponding monotone solutions of \eqref{1.6} with $c=\tilde c$ and $c=c_n$, respectively. By conclusion (1), we see $(\varphi_1,\psi_1)<(\varphi_n,\psi_n)<(\tilde\varphi,\tilde\psi)$ for $n>1$ and $x>0$. Thus $\bar\varphi:=\lim_{n\to\yy}\varphi_n$ and $\bar{\psi}:=\lim_{n\to\yy}\psi_n$ are well defined for $x\ge0$, and $(\bar{\varphi},\bar{\psi})\le(\tilde\varphi,\tilde\psi)$.

Now we claim the above convergence is in $[C^2_{\rm loc}([0,\yy))]^2$. Let us begin with proving that $\varphi'_n$ is uniformly bounded in $[0,\yy)$ for $n\ge1$.  By the first equation of \eqref{1.6}, we have
\[d_1\varphi''_n-(c_n+1)\varphi'_n-a\varphi_n+H(\psi_n)=-\varphi'_n, ~ ~ x\in(0,\yy).\]
Multiplying ${\rm e}^{-\frac{c_n+1}{d_1}x}$ to the above identity and integrating it from $x$ to $A$ yield
\bess
\varphi'_n(x)={\rm e}^{\frac{c_n+1}{d_1}(x-A)}\varphi'_n(A)-\frac1{d_1}{\rm e}^{\frac{c_n+1}{d_1}x}\int_{x}^{A}{\rm e}^{-\frac{c_n+1}{d_1}y}F_n(y)\dy
+\frac 1{d_1}{\rm e}^{\frac{c_n+1}{d_1}x}\int_{x}^{A}{\rm e}^{-\frac{c_n+1}{d_1}y}\varphi'_n(y)\dy,
\eess
where $F_n(y)=a\varphi_n-H(\psi_n)$ and $|F_n|\le M$ for some positive constant $M$ independent of $n$. Then simple calculations show
\bess
\varphi'_n(x)&\le& {\rm e}^{\frac{c_n+1}{d_1}(x-A)}\varphi'_n(A)+\frac{M}{d_1}{\rm e}^{\frac{c_n+1}{d_1}x}
\int_{x}^{A}{\rm e}^{-\frac{c_n+1}{d_1}y}\dy
+\frac 1{d_1}{\rm e}^{\frac{c_n+1}{d_1}x}\int_{x}^{A}{\rm e}^{-\frac{c_n+1}{d_1}y}\varphi'_n(y)\dy\\[1mm]
&=&{\rm e}^{\frac{c_n+1}{d_1}(x-A)}\varphi'_n(A)+\frac{M}{c_n+1}\kk(1-{\rm e}^{\frac{c_n+1}{d_1}(x-A)}\rr)\\
&&+\frac{{\rm e}^{\frac{c_n+1}{d_1}x}}{d_1}\kk({\rm e}^{-\frac{c_n+1}{d_1}y}\varphi_n(y)\Big|^{A}_x+\frac{c_n+1}{d_1}\int_{x}^{A}{\rm e}^{-\frac{c_n+1}{d_1}y}\varphi_n(y)\dy\rr)\\[1mm]
&\le& {\rm e}^{\frac{c_n+1}{d_1}(x-A)}\varphi'_n(A)+\frac{M}{c_n+1}+\frac{1}{d_1}\kk({\rm e}^{\frac{c_n+1}{d_1}(x-A)}\varphi_n-\varphi_n\rr)+\frac{u^*}{d_1}\\[1mm]
&\le&{\rm e}^{\frac{c_n+1}{d_1}(x-A)}\varphi'_n(A)+M+\frac{u^*}{d_1}{\rm e}^{\frac{c_n+1}{d_1}(x-A)}+\frac{u^*}{d_1}.
\eess
As in the proof of Lemma \ref{l2.4}, we know $\varphi'_n(x)\to0$ as $x\to\yy$. Letting $A\to\yy$ in the above inequality leads to
\bess
0<\varphi'_n(x)\le M+{u^*}/{d_1},
\eess
which implies that $\varphi'_n$ is uniformly bounded in $[0,\yy)$ for $n\ge1$. This combined with the equation  of $\varphi_n$ shows that $|\varphi''_n|$ is also  uniformly bounded in $[0,\yy)$ for $n\ge1$. Analogously, $\psi'_n$ and $|\psi''_n|$ are uniformly bounded in $[0,\yy)$ for $n\ge1$.  By differentiating \eqref{1.6} with respect to $x$, $|\varphi'''_n|$ and $|\psi'''_n|$ are also uniformly bounded in $[0,\yy)$ for $n\ge1$. By some compact considerations, the claim is verified.

Moreover, it is easy to see that $(\bar{\varphi}(\yy),\bar{\psi}(\yy))=(u^*,v^*)$. From the above analysis, it follows that $(\bar{\varphi},\bar{\psi})$ solves \eqref{1.6} with $c=\tilde{c}$. Thanks to the definition of $(\bar{\varphi},\bar{\psi})$, we know that $(\bar{\varphi},\bar{\psi})$ is nondecreasing in $x\in[0,\yy)$. Applying the strong maximum principle and the Hopf lemma to the equations satisfied by $(\bar{\varphi}',\bar{\psi}')$, we see $\bar{\varphi}'>0$ and $\bar{\psi}'>0$ for $x\ge0$. By the uniqueness, $(\bar{\varphi},\bar{\psi})$ is the unique monotone solution of \eqref{1.6} with $c=\tilde{c}$. For a sequence $\{c_n\}\subseteq[0,c^*)$ increasing to $\tilde{c}$, the details are omitted since the present case can be handled similarly (actually, it is simpler). Thus the continuity follows.

It is easy to see that the unique bounded positive solution $(U,V)$ of \eqref{1.2} is exactly the unique monotone solution of \eqref{1.6} with $c=0$. Then the convergence result \eqref{2.4} can be derived by using the similar methods as in the proof of \cite[Theorem 4.6]{WD}. We now prove the last assertion in (2). Clearly, $\ell_c$ is strictly increasing in $c\in[0,c^*)$. Then $\ell_{\yy}:=\lim_{c\to c^*}\ell_c$ is well defined.  For convenience, denote $(\Phi_c(x),\Psi_c(x))=(\varphi_c(x+\ell_c),\psi_c(x+\ell_c))$.

{\bf Claim.} $\ell_{\yy}=\yy$. In view of \eqref{1.6}, we have $(\Phi_c,\Psi_c)$ satisfies
 \bess
\left\{\!\begin{aligned}
&d_1\Phi''_c-c\Phi'_c-a\Phi_c+H(\Psi_c)=0, & &x>-\ell_c,\\
&d_2\Psi''_c-c\Psi'_c-b\Psi_c+G(\Phi_c)=0, & &x>-\ell_c.
\end{aligned}\right.
 \eess
Arguing as above, we can deduce that $\Phi'_c$, $\Psi'_c$, $|\Phi''_c|$, $|\Psi''_c|$, $|\Phi'''_c|$ and $|\Psi'''_c|$ are uniformly bounded in $[-\ell_c,\yy)$ for $c\in[0,c^*)$. If $\ell_{\yy}<\yy$, we extend $(\Phi_{c},\Psi_c)=(0,0)$ for $x\in[-\ell_{\yy},-\ell_{c}]$. Clearly, $(\Phi_{c},\Psi_c)$ is equi-continuous and uniformly bounded in $[-\ell_{\yy},\yy)$ for all $c\in[0,c^*)$. Recall $\ell_c\to\ell_\yy$ as $c\to c^*$. Then, by a compact argument and a nested subsequence method, there exists a sequence $\{c_n\}$ with $\lim_{n\to\yy}c_n=c^*$ such that $(\Phi_{c_n},\Psi_{c_n})\to(\Phi_{\yy},\Psi_{\yy})$ in $[ C^2_{\rm loc}((-\ell_{\yy},\yy))\cap C_{\rm loc}([-\ell_{\yy},\yy))]^2$. It follows that $(\Phi_{\yy},\Psi_{\yy})$ satisfies
 \bes\label{2.5}
\left\{\!\begin{aligned}
&d_1\Phi''_{\yy}-c^*\Phi'_{\yy}-a\Phi_{\yy}+H(\Psi_{\yy})=0, & &x>-\ell_{\yy},\\
&d_2\Psi''_{\yy}-c^*\Psi'_{\yy}-b\Psi_{\yy}+G(\Phi_{\yy})=0, & &x>-\ell_{\yy}.
\end{aligned}\right.
 \ees
 Since $(\Phi_{c},\Psi_c)$ is strictly increasing for $x\ge-\ell_c$, we obtain  $(\Phi_{\yy},\Psi_{\yy})$  is nondecreasing in $[-\ell_{\yy},\yy)$. Thus $\lim_{x\to-\ell_{\yy}^-}\Phi_{\yy}(x)$ and $\lim_{x\to-\ell_{\yy}^-}\Psi_{\yy}(x)$ are well defined and nonnegative. If $\lim_{x\to-\ell_{\yy}^-}\Phi_{\yy}(x)$ is positive, by the continuity and the uniform convergence, we have that for some small $\ep>0$,
 \[\Phi_{\yy}(x)\ge\frac{1}{2}\lim_{x\to-\ell_{\yy}^-}\Phi_{\yy}(x)>0 ~ ~ {\rm in ~ }[-\ell_{\yy},-\ell_{\yy}+\ep],\]
 and there exists a $N\gg1$ such that for all $n\ge N$,
 \[\Phi_{c_n}(x)\ge\frac{1}{4}\lim_{x\to-\ell_{\yy}^-}\Phi_{\yy}(x)~ ~ {\rm in ~ }[-\ell_{\yy},-\ell_{\yy}+\ep].\]
  On the other hand, we may let $n$ large enough, say $n\ge N_1>N$, such that $-\ell_{c_n}<-\ell_{\yy}+\ep$. Then \[0=\Phi_{c_n}(-\ell_{c_n})\ge\frac{1}{4}\lim_{x\to-\ell_{\yy}^-}\Phi_{\yy}(x)>0.\]
   This contradiction implies $\Phi_{\yy}(-\ell_{\yy})=0$. Similarly, $\Psi_{\yy}(-\ell_{\yy})=0$.

 By the definition of $(\Phi_{\yy},\Psi_{\yy})$, we have $\Phi_{\yy}(0)=u^*/2$. Recall $(\Phi_{\yy},\Psi_{\yy})\ge(0,0)$. Make use of the strong maximum principle, $\Phi_{\yy}(x)>0$ and $\Psi_{\yy}(x)>0$ in $x>-\ell_{\yy}$. Moreover, it follows from the above arguments that $\Phi'_{\yy}$, $\Phi''_{\yy}$, $\Psi'_{\yy}$ and $\Psi''_{\yy}$ are uniformly continuous in $[-\ell_{\yy},\yy)$. Note that $(\Phi_{\yy}(\yy),\Psi_{\yy}(\yy))$ is well defined and $(0,0)\le(\Phi_{\yy}(\yy),\Psi_{\yy}(\yy))\le(u^*,v^*)$. By Barbalat's lemma or \cite[Lemma 2.3]{WZou}, we see $(\Phi'_{\yy}(\yy),\Psi'_{\yy}(\yy))=(\Phi''_{\yy}(\yy),\Psi''_{\yy}(\yy))=(0,0)$, which together with \eqref{2.5} yields $a\Phi_{\yy}(\yy)=H(\Psi_{\yy}(\yy))$ and $b\Psi_{\yy}(\yy)=G(\Phi_{\yy}(\yy))$. This indicates $(\Phi_{\yy}(\yy),\Psi_{\yy}(\yy))=(u^*,v^*)$. In a word, $(\Phi_{\yy},\Psi_{\yy})$ satisfies
  \bess
\left\{\!\begin{aligned}
&d_1\Phi''_{\yy}-c^*\Phi'_{\yy}-a\Phi_{\yy}+H(\Psi_{\yy})=0, \quad x>-\ell_{\yy},\\
&d_2\Psi''_{\yy}-c^*\Psi'_{\yy}-b\Psi_{\yy}+G(\Phi_{\yy})=0,\,\quad x>-\ell_{\yy},\\
&\Phi_{\yy}(-\ell_{\yy})=\Psi_{\yy}(-\ell_{\yy})=0, ~ ~ \Phi_{\yy}(\yy)=u^*, ~ ~\Psi_{\yy}(\yy)=v^*.
\end{aligned}\right.
 \eess
Additionally, it is not hard to show that $\Phi'_{\yy}>0$ and $\Psi'_{\yy}>0$ in $x\ge-\ell_{\yy}$. Hence $(\Phi_{\yy}(x-\ell_{\yy}),\Psi_{\yy}(x-\ell_{\yy}))$ is a monotone solution of \eqref{1.6} with $c=c^*$, which contradicts Lemma \ref{l2.4}. Thus $\ell_{\yy}=\yy$. Our claim is proved.

Due to $\ell_{\yy}=\yy$, from the above analysis we see that  $(\Phi_{c_n},\Psi_{c_n})\to(\Phi_{\yy},\Psi_{\yy})$ in $[C^2_{\rm loc}(\mathbb{R})]^2$. Therefore, $(\Phi_{\yy},\Psi_{\yy})$ satisfies
 \bess
\left\{\!\begin{aligned}
&d_1\Phi''_{\yy}-c^*\Phi'_{\yy}-a\Phi_{\yy}+H(\Psi_{\yy})=0, & &x\in\mathbb{R},\\
&d_2\Psi''_{\yy}-c^*\Psi'_{\yy}-b\Psi_{\yy}+G(\Phi_{\yy})=0, & &x\in\mathbb{R}.
\end{aligned}\right.
 \eess
Moreover, similar to the above, it can be deduced that $\Phi_{\yy}, \Psi_{\yy}, \Phi', \Psi'_{\yy}>0$ in $\mathbb{R}$. Recall that $(0,0)\le(\Phi_{\yy},\Psi_{\yy})\le(u^*,v^*)$ in $\mathbb{R}$. Then  $\Phi_{\yy}(\pm\yy)$ and $\Psi_{\yy}(\pm\yy)$ are well defined. Notice that  $\Phi'_{\yy}$, $\Phi''_{\yy}$, $\Psi'_{\yy}$ and $\Psi''_{\yy}$ are uniformly continuous in $\mathbb{R}$, it follows that $a\Phi_{\yy}(\pm\yy)=H(\Psi_{\yy}(\pm\yy))$ and $b\Psi_{\yy}(\pm\yy)=G(\Phi_{\yy}(\pm\yy))$, which combined with {\bf (H)} leads to $(\Phi_{\yy}(-\yy),\Psi_{\yy}(-\yy))=(0,0)$ and $(\Phi_{\yy}(\yy),\Psi_{\yy}(\yy))=(u^*,v^*)$. Therefore, $(\Phi_{\yy},\Psi_{\yy})$ is the travelling wave solution of \eqref{1a.2} with speed $c^*$. This implies conclusion (2).

(3) For any given $\mu_1>0$ and $\mu_2>0$, we define
 \[f(c)=\mu_1\varphi'_c(0)+\mu_2\psi'_c(0)-c, ~ ~ c\in[0,c^*).\]
By conclusion (2), $f(c)$ is continuous in $[0,c^*)$. As in (1), we have $\varphi'_{c_1}(0)>\varphi'_{c_2}(0)$ and $\psi'_{c_1}(0)>\psi'_{c_2}(0)$ for any $0\le c_1<c_2<c^*$ by applying the Hopf lemma to the equations of $\varphi_{c_1}-\varphi_{c_2}$ and $\psi_{c_1}-\psi_{c_2}$, respectively. Hence $f(c)$ is strictly decreasing in $c\in[0,c^*)$. Moreover, $f(0)>0$ and $\lim_{c\to c^*}f(c)=-c^*<0$ by the second limit of \qq{2.4}. Thus there exists a unique $c_{\mu_1,\mu_2}\in(0,c^*)$ such that $f(c_{\mu_1,\mu_2})=0$.

To show the limits of $c_{\mu_1,\mu_2}$, we rewrite $f(c)$ as $f(c,\mu_1,\mu_2)$. Clearly, $f(c,\mu_1,\mu_2)$ is strictly decreasing in $c\in[0,c^*)$ and strictly increasing in $\mu_i>0$ for $i=1,2$, which implies that $c_{\mu_1,\mu_2}$ is strictly increasing in $\mu_i>0$ for $i=1,2$. Fix $\mu_2>0$. For any small $\ep>0$, it is easy to see that $f(c^*-\ep,\mu_1,\mu_2)\to\yy$ as $\mu_1\to\yy$.
As $f(c,\mu_1,\mu_2)$ is strictly decreasing in $c\in[0,c^*)$ and $f(c_{\mu_1,\mu_2},\mu_1,\mu_2)=0$, it yields $c^*-\ep<c_{\mu_1,\mu_2}<c^*$ for all large $\mu_1$. Thus $c_{\mu_1,\mu_2}\to c^*$ as $\mu_1\to\yy$.

On the other hand, for any small $\ep>0$, $f(\ep,\mu_1,\mu_2)\to-\ep<0$ as $\mu_1+\mu_2\to0$, which combined with the monotonicity shows that $c_{\mu_1,\mu_2}<\ep$ if $\mu_1+\mu_2$ is small enough.
 The conclusion (3) is obtained and the proof is finished.
 \end{proof}
Clearly, Theorem \ref{t1.1} follows from Lemmas \ref{l2.4}-\ref{l2.6}.

\begin{remark}\label{r2.1}Taking advantage of the similar arguments in the above proof, one easily shows that the unique monotone solution $(\varphi,\psi)$ of \eqref{1.6} depends continuously on the parameters in \eqref{1.6}. This observation will be used to perturb \eqref{1.6} to construct some suitable upper and lower solutions.
\end{remark}
\section{Spreading speed of \eqref{1a.5}}

In this section, by using the semi-wave problem \eqref{1.6} we determine the spreading speed of \eqref{1a.5} when spreading happens, namely, Theorem \ref{t1.2}. The process will be divided into two cases, case 1: $\mathbb{B}[w]=w$ and case 2: $\mathbb{B}[w]=w_x$. The following comparison principle will be used later. Since its proof is similar to \cite[Lemma 2.3]{WD}, the details are omitted.

\begin{lem}[Comparison principle]\label{l3.0} Let $s(t)\in C^1([0,T])$, $s(t)\ge0$ in $[0,T]$, and $(\bar{u},\bar{v},\bar{h})\in [C^{1,2}(\Omega_T)\cap C(\overline{\Omega}_T)]^2\times C^1([0,T])$ for $T>0$, and satisfy
\bes\begin{cases}\label{3.0}
\bar u_t\ge d_1\bar u_{xx}-a\bar u+H(\bar v), &0<t\le T,~x\in(s(t),\bar h(t)),\\
\bar v_t\ge d_2\bar v_{xx}-b\bar v+G(\bar u), &0<t\le T,~x\in(s(t),\bar h(t)),\\
\bar u(t,s(t))\ge u(t,s(t)), ~ \bar v(t,s(t))\ge v(t,s(t)), &0<t\le T,\\
\bar u(t,\bar h(t))\ge0, ~ \bar v(t,\bar h(t))\ge0, &0<t\le T,\\
\bar h'(t)\ge-\mu_1 \bar u_x(t,\bar h(t))-\mu_2\bar v_x(t,\bar h(t)), & 0<t\le T,\\
\bar h(0)\ge h_0, ~ \bar u(0,x)\ge u_{0}(x), ~ \bar v(0,x)\ge v_0(x),&s(0)\le x\le h_0,
 \end{cases}
 \ees
 where $\Omega_T=\{(t,x):0<t\le T, ~ s(t)<x<\bar{h}(t)\}$. Then the unique solution $(u,v,h)$ of \eqref{1.6} satisfies
 \[h(t)\le \bar{h}(t), ~ u(t,x)\le \bar{u}(t,x),~ v(t,x)\le \bar{v}(t,x) ~ {\rm for ~ }t\in[0,T], ~ x\in[s(t),h(t)].\]
\end{lem}

We usually call $(\bar{u},\bar{v},\bar{h})$ in the above lemma an upper solution for \eqref{1a.5}. If we reverse all the inequalities in \eqref{3.0}, then we can define a lower solution.

\subsection{The spreading speed in case 1: operator $\mathbb{B}[w]=w$}

\begin{lem}\label{l3.1}Let $(u,v,h)$ be the unique solution of \eqref{1a.5} with operator $\mathbb{B}[w]=w$. Then
\bess
\limsup_{t\to\yy}\frac{h(t)}{t}\le c_{\mu_1,\mu_2}.
\eess
\end{lem}

\begin{proof}
We first perturb the semi-wave problem \eqref{1.6}. Due to $\mathcal{R}_0>1$ and condition {\bf (H)}, there exists a small $\ep_1>0$ such that system $(a-\ep)u=H(v)$  and $(b-\ep)v=G(u)$ has a unique positive root $(u^*_{\ep},v^*_{\ep})$ when $\ep\in[0,\ep_1)$. Clearly, $u^*_{\ep}$ and $v^*_{\ep}$ are strictly increasing in $\ep\in[0,\ep_1)$. Consider the ODE system
  \bess\begin{cases}
\hat u_t=-a\hat u+H(\hat v),\\
\hat v_t=-b\hat v+G(\hat u),\\
\hat u(0)=\|u_0\|_{C([0,h_0])}, ~  \hat v(0)=\|v_0\|_{C([0,h_0])}.
 \end{cases}
  \eess
Since $\mathcal{R}_0>1$, $(\hat{u}(t),\hat{v}(t))\to(u^*,v^*)$ as $t\to\yy$.
By a simple comparison consideration, we have
$(u(t,x),v(t,x))\le(\hat{u}(t),\hat{v}(t))$ for $t\ge0$ and $x\ge0$. This  indicates
 \bes
 \limsup_{t\to\yy}(u(t,x),v(t,x))\le(u^*,v^*) ~ ~ {\rm uniformly ~ in ~ }[0,\yy).\lbl{3.2a}\ees
 Thus there exists a $T>0$ such that
 \[(u(t,x),v(t,x))\le (u^*_{\ep},v^*_{\ep}) ~ ~{\rm for ~ }t\ge T, ~ x\ge0.\]
Let $2\ep\in[0,\ep_1)$ and consider the following perturbed problem of \eqref{1.6}
 \bess
\left\{\!\begin{aligned}
&d_1\varphi''-c\varphi'-(a-2\ep)\varphi+H(\psi)=0, \quad x>0,\\
&d_2\psi''-c\psi'-(b-2\ep)\psi+G(\varphi)=0, \quad\,x>0,\\
&\varphi(0)=\psi(0)=0, ~ \varphi(\yy)=u^*_{2\ep}, ~ \psi(\yy)=v^*_{2\ep}.
\end{aligned}\right.
 \eess
By Theorem \ref{t1.1}, the above problem has a unique monotone solution $(\varphi_{2\ep},\psi_{2\ep})$ and there exists a unique $c_{\mu_1,\mu_2,2\ep}>0$ such that
 \[c_{\mu_1,\mu_2,2\ep}=\mu_1\varphi'_{2\ep}(0)+\mu_2\psi'_{2\ep}(0).\]
By Remark \ref{r2.1}, $c_{\mu_1,\mu_2,2\ep}\to c_{\mu_1,\mu_2}$ as $\ep\to0$. Moreover, there exists a unique $X>0$ such that $(\varphi_{2\ep}(X),\psi_{2\ep}(X))>(u^*_{\ep},v^*_{\ep})$ as $(u^*_{2\ep},v^*_{2\ep})>(u^*_{\ep},v^*_{\ep})$.

Choose $L>h(T)$ and define
 \[\bar{h}(t)=c_{\mu_1,\mu_2,2\ep}(t-T)+X+L, ~ ~ \bar{u}=\varphi_{2\ep}(\bar{h}(t)-x), ~ ~ \bar{v}=\psi_{2\ep}(\bar{h}(t)-x)\]
for $t\ge T$ and $x\in[0,\bar{h}(t)]$. Next we show that $(\bar{u},\bar{v},\bar{h})$ satisfies
 \bes\begin{cases}\label{3.1}
\bar u_t\ge d_1\bar u_{xx}-a\bar u+H(\bar v), &t>T,~x\in(0,\bar h(t)),\\
 \bar v_t\ge d_2\bar v_{xx}-b\bar v+G(\bar u), &t>T,~x\in(0,\bar h(t)),\\
\bar u(t,0)\ge0, ~ \bar v(t,0)\ge0, ~ \bar u(t,\bar h(t))\ge0, ~ \bar v(t,\bar h(t))\ge0, \; &t>T,\\
\bar h'(t)\ge-\mu_1 \bar u_x(t,\bar h(t))-\mu_2\bar v_x(t,\bar h(t)), & t>T,\\
\bar h(T)\ge h(T), ~ \bar u(T,x)\ge u(T,x), ~ \bar v(T,x)\ge v(T,x),&0\le x\le h(T).
 \end{cases}
 \ees
Once it is done, by a comparison argument (Lemma \ref{l3.0}), we have $h(t)\le\bar{h}(t)$ for $t\ge T$, which, combined with the definition of $\bar{h}(t)$, implies
 \[\limsup_{t\to\yy}\frac{h(t)}{t}\le \lim_{t\to\yy}\frac{\bar h(t)}{t}=c_{\mu_1,\mu_2,2\ep}.\]
Since $c_{\mu_1,\mu_2,2\ep}\to c_{\mu_1,\mu_2}$ as $\ep\to0$, we obtain the desired result. So it remains to verify \eqref{3.1}.

Firstly, the inequalities in the third line of \eqref{3.1} are obvious. Simple computations show
 \bess
 -\mu_1 \bar u_x(t,\bar h(t))-\mu_2\bar v_x(t,\bar h(t))=\mu_1\varphi'_{2\ep}(0)+\mu_2\psi'_{2\ep}(0)=c_{\mu_1,\mu_2,2\ep}=\bar{h}'(t), ~ ~{\rm for ~ }t\ge T.
 \eess
 So the inequality in the fourth line of \eqref{3.1} holds. By the definition of $\bar{h}$, we see $\bar{h}(T)>h(T)$. Moreover, for $x\in[0,h(T)]$,
 \[(\bar{u}(T,x),\bar{v}(T,x))\ge(\varphi_{2\ep}(X),\psi_{2\ep}(X))\ge(u^*_{\ep},v^*_{\ep})\ge(u(T,x),v(T,x)),\]
which implies that the inequalities in the fifth line are valid.

 Straightforward calculations yield that for $t>T$ and $x\in(0,\bar h(t))$,
 \bess
 \bar u_t-d_1\bar u_{xx}+a\bar u-H(\bar v)=2\ep\bar{u}\ge0,\\
  \bar v_t-d_2\bar v_{xx}+b\bar v-G(\bar u)=2\ep\bar{v}\ge0.
 \eess
 Therefore, \eqref{3.1} is obtained and the proof is ended.
\end{proof}

Now we show the lower limit of $h(t)/t$ as $t\to\yy$. When operator $\mathbb{B}[w]=w$ and spreading happens, we know that solution component $(u,v)$ converges to a non-constant steady state solution $(U,V)$, which bring some difficulties for the construction of the lower solution.

\begin{lem}\label{l3.2}Let $(u,v,h)$ be the unique solution of \eqref{1a.5} with operator $\mathbb{B}[w]=w$. Then
\bess
\liminf_{t\to\yy}\frac{h(t)}{t}\ge c_{\mu_1,\mu_2}.
\eess
\end{lem}
\begin{proof}
In view of $\mathcal{R}_0>1$ and condition {\bf (H)}, there exists a small $\delta_0>0$ such that system $(a+\delta)u=H(v)$ and $(b+\delta)v=G(u)$ has a unique positive root $(u^*_{\delta},v^*_{\delta})$ if $\delta\in[0,\delta_0)$. It is easy to see that $(u^*_{\delta},v^*_{\delta})$ is strictly decreasing in $\delta\in[0,\delta_0)$ and $(u^*_{\delta},v^*_{\delta})\to(u^*,v^*)$ as $\delta\to0$. Let $2\delta\in[0,\delta_0)$ and $(U,V)$ be the unique bounded positive solution of \eqref{1.2}. From the proof of \cite[Lemma 2.3]{LL}, we know $(U,V)$ strictly increases to $(u^*,v^*)$ as $x\to\yy$. Thus there is a $X_{\delta}\gg1$ such that $(U(X_{\delta}),V(X_{\delta}))>(u^*_{{\delta}/{2}},v^*_{{\delta}/{2}})$. Since $(u,v)\to(U,V)$ in $C_{\rm loc}([0,\yy))$ as $t\to\yy$, we can find a $T\gg1$ such that $h(T)\ge X_{\delta}+1$ and
\[(u(t,x),v(t,x))\ge(u^*_{\delta},v^*_{\delta}) ~ ~ {\rm for ~ }t\ge T, ~ x\in[X_{\delta},X_{\delta}+1].\]
Let $(\varphi_{2\delta},\psi_{2\delta})$ be the unique monotone solution of
 \bess
\left\{\!\begin{aligned}
&d_1\varphi''-c\varphi'-(a+2\delta)\varphi+H(\psi)=0, \;\;x>0,\\
&d_2\psi''-c\psi'-(b+2\delta)\psi+G(\varphi)=0, \;\;\,x>0,\\
&\varphi(0)=\psi(0)=0, ~ \varphi(\yy)=u^*_{2\delta}, ~ \psi(\yy)=v^*_{2\delta}.
\end{aligned}\right.
 \eess
 By Theorem \ref{t1.1}, there exists a unique $c_{\mu_1,\mu_2,2\delta}>0$ such that
 \[c_{\mu_1,\mu_2,2\delta}=\mu_1\varphi'_{2\delta}(0)+\mu_2\psi'_{2\delta}(0).\]
  Moreover, using Remark \ref{r2.1} we have $c_{\mu_1,\mu_2,2\delta}\to c_{\mu_1,\mu_2}$ as $\delta\to0$.

Define
 \[\ud h(t)=c_{\mu_1,\mu_2,2\delta}(t-T)+X_{\delta}+1, ~ ~ \ud u=\varphi_{2\delta}(\ud h(t)-x), ~ ~ \ud v=\psi_{2\delta}(\ud h(t)-x)\]
for $t\ge T$ and $x\in[X_{\delta},\ud h(t)]$. We next prove the following inequalities
 \bes\begin{cases}\label{3.2}
\ud u_t\le d_1\ud u_{xx}-a\ud u+H(\ud v), &t>T,~x\in(X_{\delta},\bar h(t)),\\
 \ud v_t\le d_2\ud v_{xx}-b\ud v+G(\ud u), &t>T,~x\in(X_{\delta},\bar h(t)),\\
\ud u(t,X_{\delta})\le u(t,X_{\delta}), ~ \ud v(t,X_{\delta})\le v(t,X_{\delta}), &t>T,\\
\ud u(t,\ud h(t))\le0, ~ \ud v(t,\ud h(t))\le0, &t>T,\\
\ud h'(t)\le-\mu_1 \ud u_x(t,\ud h(t))-\mu_2\ud v_x(t,\ud h(t)), & t>T,\\
\ud h(T)\le h(T), ~ \ud u(T,x)\le u(T,x), ~ \ud v(T,x)\le v(T,x),&X_{\delta}\le x\le \ud h(T).
 \end{cases}
 \ees
 Once we have done that, using a comparison argument (Lemma \ref{l3.0} with $s(t)\equiv X_{\delta}$), we get $h(t)\ge\ud h(t)$ for $t\ge T$, which indicates
 \[\liminf_{t\to\yy}\frac{h(t)}{t}\ge\lim_{t\to\yy}\frac{\ud h(t)}{t}=c_{\mu_1,\mu_2,2\delta}.\]
 Recall that $c_{\mu_1,\mu_2,2\delta}\to c_{\mu_1,\mu_2}$ as $\delta\to0$. We immediately obtain the result as wanted.

 Now let us begin with verifying \eqref{3.2}. Firstly, thanks to our choice of $T$, we have $\ud h(T)=X_{\delta}+1\le h(T)$, and for $(t,x)\in[T,\yy)\times[X_{\delta},X_{\delta}+1]$,
 \[(\ud u(t,x),\ud v(t,x))\le(u^*_{2\delta},v^*_{2\delta})\le(u^*_{\delta},v^*_{\delta})\le(u(t,x),v(t,x)),\]
 which clearly implies the inequalities in the third, fourth and sixth lines of \eqref{3.2} are valid. The direct computations show
 \[-\mu_1 \ud u_x(t,\ud h(t))-\mu_2\ud v_x(t,\ud h(t))=\mu_1\varphi'_{2\delta}(0)+\mu_2\psi'_{2\delta}(0)=c_{\mu_1,\mu_2,2\delta}=\ud h'(t).\]
 Moreover, it is not hard to deduce that for $(t,x)\in[T,\yy)\times(X_{\delta},\ud h(t))$,
 \bess
 &&\ud u_t-d_1\ud u_{xx}+a\ud u-H(\ud v)=-2\delta\ud u\le0,\\
 &&\ud v_t-d_2\ud v_{xx}+b\ud v-G(\ud u)=-2\delta\ud v\le0.
 \eess
 Therefore, \eqref{3.2} holds and the proof is complete.
\end{proof}

Next we show the asymptotical behavior of $(u,v)$ as $t\to\yy$. Firstly, we need a detailed understanding for the steady state problem \eqref{1.2}. In fact, as we can see from Theorem \ref{t1.1}, the unique bounded positive solution $(U,V)$ of problem \eqref{1.2} actually is the unique monotone solution of the semi-wave problem \eqref{1.6} with $c=0$. Consider the following perturbed problem of \eqref{1.2}
\bes\begin{cases}\label{3.3}
-d_1u''=-(a+\ep)u+H(v), \;\;&x\in(0,\yy),\\
-d_2v''=-(b+\ep)v+G(u), &x\in(0,\yy),\\
u(0)=v(0)=0.
 \end{cases}
\ees

\begin{lem}\label{l3.3}There exists a small $\ep_0>0$ such that problem \eqref{3.3} has a unique bounded positive solution $(U_{\ep},V_{\ep})$ if $\ep\in(-\ep_0,\ep_0)$, and $(U_{\ep},V_{\ep})$ is strictly increasing in $x\ge0$ and $(U_{\ep}(\yy),V_{\ep}(\yy))=(u^*_{\ep},v^*_{\ep})$, where $(u^*_{\ep},v^*_{\ep})$ is the unique positive root of
 \bess
 (a+\ep)u=H(v),\;\;\;(b+\ep)v=G(u).\eess

Moreover, $(U_{\ep},V_{\ep})$ is strictly decreasing for $\ep\in(-\ep_0,\ep_0)$, i.e., $(U_{\ep_1}(x),V_{\ep_1}(x))>(U_{\ep_2}(x),V_{\ep_2}(x))$ for $x>0$ and any $-\ep_0<\ep_1<\ep_2<\ep_0$, and $(U_{\ep},V_{\ep})\to(U,V)$ in $L^{\yy}([0,\yy))$ as $\ep\to0$, where $(U,V)$ is the unique bounded positive solution of \eqref{1.2}.
\end{lem}

\begin{proof} According to $\mathcal{R}_0>1$ and condition {\bf (H)}, by arguing as in the proof of \cite[Lemma 2.3]{LL} or Theorem \ref{t1.1}, we can find a small $\ep_0>0$ such that \eqref{3.3} has a unique bounded positive solution $(U_{\ep},V_{\ep})$ for $\ep\in(-\ep_0,\ep_0)$. The monotonicity and the limits as $x\to\yy$ also can be obtained similarly.

We now prove the monotonicity with respect to $\ep\in(-\ep_0,\ep_0)$. Clearly, $(U_{\ep_2}(x),V_{\ep_2}(x))$ satisfies
\bess\begin{cases}
-d_1U''_{\ep_2}<-(a+\ep_1)U_{\ep_2}+H(V_{\ep_2}), \;\;x\in(0,\yy),\\
-d_2V''_{\ep_2}<-(b+\ep_1)V_{\ep_2}+G(U_{\ep_2}), \;\;\,x\in(0,\yy),\\
U_{\ep_2}(0)=V_{\ep_2}(0)=0, ~U_{\ep_2}(\yy)=u^*_{\ep_2}<u^*_{\ep_1}, ~ V_{\ep_2}(\yy)=v^*_{\ep_2}<v^*_{\ep_1}.
 \end{cases}
\eess
Moreover, as in Lemma \ref{l2.2}, we can let $k$ be sufficiently large such that $(\bar{\varphi},\bar{\psi})$ and $(U_{\ep_2}(x),V_{\ep_2}(x))$ are the upper and lower solutions for \eqref{3.3} with $\ep=\ep_1$ and
\bess
 \sup_{0\le x\le y}U_{\ep_2}(x)\le\bar{\varphi}(y), ~ \sup_{0\le x\le y}V_{\ep_2}(x)\le\bar{\psi}(y), ~ ~ \forall\, ~ y\ge0,
 \eess
where $(\bar{\varphi},\bar{\psi})$ is defined in Lemma \ref{l2.2}.
Making use of \cite[Proposition 2.6]{WND} we have that the problem \eqref{3.3} with $\ep=\ep_1$ has a monotone solution $(\tilde U_{\ep_1},\tilde V_{\ep_1})$ and $(U_{\ep_2},V_{\ep_2})\le (\tilde U_{\ep_1},\tilde V_{\ep_1})\le (\bar{\varphi},\bar{\psi})$. By the uniqueness, $(\tilde U_{\ep_1},\tilde V_{\ep_1})=(U_{\ep_1},V_{\ep_1})$. Then, applying the strong maximum principle to the equations of $U_{\ep_1}-U_{\ep_2}$ and $V_{\ep_1}-V_{\ep_2}$, it can be derived that $(U_{\ep_1}, V_{\ep_1})>(U_{\ep_2},V_{\ep_2})$.

To prove $(U_{\ep},V_{\ep})\to(U,V)$ in $L^{\yy}([0,\yy))$ as $\ep\to0$, it suffices to show that for any small $\delta>0$, there exists a $\ep_1<\ep_0$ such that
\bes\label{3.5}(U(x)-\delta,V(x)-\delta)\le(U_{\ep}(x),V_{\ep}(x))\le(U(x)+\delta,V(x)+\delta), ~  ~ \forall\, x\ge0, ~ \ep\in(-\ep_1,\ep_1).\ees
Let $\ep\in(0,\ep_0)$. Define $f_{\ep}(x)=U_{\ep}(x)-U(x)$ for $x\ge0$. Obviously, $f_{\ep}(\yy)=u^*_{\ep}-u^*<0$ and $u^*_{\ep}-u^*\to0$ as $\ep\to0$. Thus there exists a small $\ep_2<\ep_0$ such that $f_{\ep_2}(\yy)>-\delta/2$. By continuity, there exists a large $X>0$ such that $f_{\ep_2}(x)\ge-\delta$ for $x\ge X$. Note that $X$ depends only on $\delta$. Since $f_{\ep}(x)$ is strictly decreasing in $\ep$, we have $f_{\ep}(x)\ge-\delta$ for $x\ge X$ and $0<\ep<\ep_2$. Using the analogous methods as in the proof of Lemma \ref{l2.6}, we have that the map $\ep\longmapsto(U_{\ep},V_{\ep})$ is continuous from $(-\ep_0,\ep_0)$ to $[C^2_{\rm loc}([0,\yy))]^2$. So there exists a small $0<\ep_3<\ep_2$ such that  $f_{\ep}(x)\ge-\delta$ for $x\in[0,X]$ and $\ep<\ep_3$. Similarly, we also can show that there exists a small $\ep_4$ such that $V(x)-\delta\le V_{\ep}(x)$ for $x\ge0$ and $\ep<\ep_4$. Thus the left side of \eqref{3.5} is obtained since the case $\ep\in(-\ep_0,0)$ is obvious. The right side of \eqref{3.5} can be deduced similarly, and we omitted the details. The proof is finished.
\end{proof}

\begin{lem}\label{l3.4}Let $(U,V)$ be the unique bounded positive solution of \eqref{1.2} and $(u,v,h)$ be the unique solution of \eqref{1a.5} with operator $\mathbb{B}[w]=w$. Then
\bess
\lim_{t\to \yy}\max_{x\in[0,ct]}\big(|u(t,x)-U(x)|+|v(t,x)-V(x)|\big)=0, ~ \forall\, c\in[0,c_{\mu_1,\mu_2}).
\eess
\end{lem}
\begin{proof}We complete the proof by two steps.

{\bf Step 1.} In this step, we show that for any $c\in[0,c_{\mu_1,\mu_2})$,
\[\liminf_{t\to\yy}(u(t,x),v(t,x))\ge(U(x),V(x)) ~ ~ {\rm uniformly ~ in ~ }[0,ct].\]
To this end, it is sufficient to prove that for any $\ep>0$, there exists $T>0$ such that
 \bes
 (u(t,x),v(t,x))\ge(U(x)-\ep,V(x)-\ep) ~ ~ {\rm for ~ }t\ge T,\; x\in[0,ct],
 \lbl{3.6a}\ees
which will be obtained by using the lower solution constructed in the proof of Lemma \ref{l3.2}. For convenience, we recall that this lower solution $(\ud u,\ud v,\ud h)$ was defined by
\[\ud h(t)=c_{\mu_1,\mu_2,2\delta}(t-T)+X_{\delta}+1, ~ ~ \ud u=\varphi_{2\delta}(\ud h(t)-x), ~ ~ \ud v=\psi_{2\delta}(\ud h(t)-x),\]
where $\delta>0$ is small enough such that system $(a+2\delta)u=H(v)$ and $(b+2\delta)v=G(u)$ has a unique positive root $(u^*_{2\delta},\,v^*_{2\delta})$ and $X_{\delta}$ is determined by $(U(X_{\delta}),V(X_{\delta}))>(u^*_{{\delta}/{2}},v^*_{{\delta}/{2}})$.

For any $\ep>0$, by Lemma \ref{l3.3}, we can choose $\delta$ to be further small such that
\[(U_{2\delta}(x),V_{2\delta}(x))\ge\kk(U(x)-{\ep}/{2},\,V(x)-{\ep}/{2}\rr)~ ~ {\rm for ~ } x\ge0.\]
Moreover, due to the proof of Lemma \ref{l3.2}, for any $c\in[0,c_{\mu_1,\mu_2})$, we can again shrink $\delta$ such that $c_{\mu_1,\mu_2,2\delta}>c$. On the other hand, by enlarging $X_{\delta}$, we have
 \[(U_{2\delta}(X_{\delta}),V_{2\delta}(X_{\delta}))\ge
 \kk(u^*_{2\delta}-{\ep}/{4},\,v^*_{2\delta}-{\ep}/{4}\rr).\]
Then, as in the proof of Lemma \ref{l3.2}, there exists $T>0$ such that
  \[(u(t,x),v(t,x))\ge(\ud u(t,x),\ud v(t,x)) ~ ~ {\rm and ~ ~ } h(t)\ge \ud h(t) ~ ~ {\rm for ~ }t\ge T, ~ x\in[X_{\delta},\ud h(t)].\]

{\bf Claim.} There exists a $T_1>T$ such that
  \[(\ud u(t,x),\ud v(t,x))\ge\kk(U_{2\delta}(x)-{\ep}/{2},\,
  V_{2\delta}(x)-{\ep}/{2}\rr)~ ~ {\rm for ~ }\, t\ge T_1,\;x\in[X_{\delta},ct].\]
In fact, by the definition of the lower solution and the choice of $X_{\delta}$, there is a $T_1>T$ such that
\bess
&&\max_{x\in[X_{\delta},ct]}|\ud u(t,x)-U_{2\delta}(x)|\le u^*_{2\delta}-\ud u(t,ct)+u^*_{2\delta}-U_{2\delta}(X_{\delta})\le{\ep}/{2} ~ ~ {\rm for  ~ }t\ge T_1,\\
&&\max_{x\in[X_{\delta},ct]}|\ud v(t,x)-V_{2\delta}(x)|\le v^*_{2\delta}-\ud v(t,ct)+v^*_{2\delta}-V_{2\delta}(X_{\delta})\le{\ep}/{2} ~ ~ {\rm for  ~ }t\ge T_1,
\eess
which clearly implies our claim.

By the above claim, we have that, for $t\ge T_1$ and $x\in[X_{\delta},ct]$,
  \[(u(t,x),v(t,x))\ge(\ud u(t,x),\ud v(t,x))\ge\kk(U_{2\delta}(x)-{\ep}/{2},\,
  V_{2\delta}(x)-{\ep}/{2}\rr)\ge(U(x)-\ep,V(x)-\ep).\]
Since spreading happens, one can find a $T_2>T_1$ such that
\[(u(t,x),v(t,x))\ge(U(x)-\ep,V(x)-\ep) ~ ~ {\rm for ~ }\;T_2>T_1,\; x\in[0,X_{\delta}].\]
All in all, for any $\ep>0$, when $t\ge T_2$, the inequality \qq{3.6a} holds.

{\bf Step 2.} In this step, we show
\[\limsup_{t\to\yy}(u(t,x),v(t,x))\le(U(x),V(x)) ~ ~ {\rm uniformly ~ in ~ }[0,\yy).\]
Certainly, it suffices to prove that for any $\ep>0$, there exists a $T>0$ such that
  \bes
  (u(t,x),v(t,x))\le(U(x)+\ep,V(x)+\ep) ~ ~ {\rm for ~ }\,t\ge T,\; x\in[0,\yy).
  \lbl{3.7a}\ees
Since $(U(x)+\ep,V(x)+\ep)\to(u^*+\ep,v^*+\ep)$ as $x\to\yy$, one can find a $X>0$ such that
$(U(x)+\ep,V(x)+\ep)\ge\kk(u^*+\ep/2,\,v^*+\ep/2\rr)$ for $x\ge X$.
Moreover,  in the proof of Lemma \ref{l3.1} we have obtained \qq{3.2a},
which implies that there exists a $T_1>0$ such that
$(u(t,x),v(t,x))\le\kk(u^*+\ep/4,\,v^*+\ep/4\rr)$ for $t\ge T_1$ and  $x\ge0$. Thus we have
  \[(u(t,x),v(t,x))\le(U(x)+\ep,V(x)+\ep) ~ ~ {\rm for ~ }\,t\ge T_1,\; x\in[X,\yy).\]
On the other hand, since $(u,v)\to(U,V)$ in $C_{\rm loc}([0,\yy))$ as $t\to\yy$, there is a $T_2>T_1$ such that
  \[(u(t,x),v(t,x))\le(U(x)+\ep,V(x)+\ep) ~ ~ {\rm for ~ }\,t\ge T_2,\; x\in[0,X].\]
It follows that, when $t\ge T_2$, the inequality \qq{3.7a} holds. Therefore, the proof is complete.
\end{proof}

\subsection{The spreading speed in case 2: operator $\mathbb{B}[w]=w_x$}
Now we are going to prove the spreading speed of \eqref{1a.5} when the fixed boundary $x=0$ is subject to the homogeneous Neumann condition, i.e., operator $\mathbb{B}[w]=w_x$. This also will be achieved by several lemmas.

\begin{lem}\label{l3.5}Let $(u,v,h)$ be the unique solution of \eqref{1a.5} with operator $\mathbb{B}[w]=w_x$. Then
\[\limsup_{t\to\yy}\frac{h(t)}{t}\le c_{\mu_1,\mu_2}.\]
\end{lem}
\begin{proof}Actually, as in the proof of Lemma \ref{l3.1}, it is easy to check that there is a $T>0$ such that $(\bar{u},\bar{v},\bar{h})$ defined in Lemma \ref{l3.1} satisfies
\bess\begin{cases}
\bar u_t\ge d_1\bar u_{xx}-a\bar u+H(\bar v), &t>T,~x\in(0,\bar h(t)),\\
 \bar v_t\ge d_2\bar v_{xx}-b\bar v+G(\bar u), &t>T,~x\in(0,\bar h(t)),\\
\bar u(t,0)\ge u(t,0), ~ \bar v(t,0)\ge v(t,0), ~ \bar u(t,\bar h(t))\ge0, ~ \bar v(t,\bar h(t))\ge0, \; &t>T,\\
\bar h'(t)\ge-\mu_1 \bar u_x(t,\bar h(t))-\mu_2\bar v_x(t,\bar h(t)), & t>T,\\
\bar h(T)\ge h(T), ~ \bar u(T,x)\ge u(T,x), ~ \bar v(T,x)\ge v(T,x),&0\le x\le h(T).
 \end{cases}
 \eess
Using Lemma \ref{l3.0}, we have $\bar{h}(t)\ge h(t)$, which leads to our desired result. The proof is complete.
\end{proof}

Then we turn to the lower limit of $h(t)/t$ and the asymptotical behavior of the solution component $(u,v)$.

\begin{lem}\label{l3.6}Let $(u,v,h)$ be the unique solution of \eqref{1a.5} with operator $\mathbb{B}[w]=w_x$. Then
 \bes
&\dd\liminf_{t\to\yy}\frac{h(t)}{t}\ge c_{\mu_1,\mu_2},&\lbl{3.8}\\[2mm]
&\dd\lim_{t\to \yy}\max_{x\in[0,ct]}\big(|u(t,x)-u^*|+|v(t,x)-v^*|\big)=0, ~ \forall\, c\in[0,c_{\mu_1,\mu_2}).&\lbl{3.9}
 \ees
\end{lem}

\begin{proof}Similar to the proof of Lemma \ref{l3.2}, we perturb the semi-wave problem of \eqref{1.6} and consider
\bess
\left\{\!\begin{aligned}
&d_1\varphi''-c\varphi'-(a+2\delta)\varphi+H(\psi)=0, \;\;x>0,\\
&d_2\psi''-c\psi'-(b+2\delta)\psi+G(\varphi)=0, \;\;\,x>0,\\
&\varphi(0)=\psi(0)=0, ~ \varphi(\yy)=u^*_{2\delta}, ~ \psi(\yy)=v^*_{2\delta},
\end{aligned}\right.
 \eess
 where $\delta$ is sufficiently small such that system $(a+2\delta)u=H(v)$ and $(b+2\delta)v=G(u)$ has a unique positive root $(u^*_{2\delta},v^*_{2\delta})$. Let $(\varphi_{2\delta},\psi_{2\delta})$ be the unique monotone solution of this perturbation problem.
Making use of Theorem \ref{t1.1} and Remark \ref{r2.1}, there exists a unique $c_{\mu_1,\mu_2,\delta}>0$ such that \[c_{\mu_1,\mu_2,2\delta}=\mu_1\varphi'_{2\delta}(0)+\mu_2\psi'_{2\delta}(0)\]
and $c_{\mu_1,\mu_2,2\delta}\to c_{\mu_1,\mu_2}$ as $\delta\to0$.

Since spreading occurs, we have $(u(t,x),v(t,x))\to(u^*,v^*)$ in $C_{\rm loc}([0,\yy))$ as $t\to\yy$. There exist $L>0$ and $T>0$ such that $h(T)>L$ and
\[(u(t,x),v(t,x))\ge(u^*_{\delta},v^*_{\delta}) ~ ~{\rm for ~ }t\ge T, ~ x\in[0,L].\]
Define
  \[\ud h(t)=c_{\mu_1,\mu_2,2\delta}(t-T)+L, ~ \ud u=\varphi_{2\delta}(\ud h(t)-x), ~ \ud v=\psi_{2\delta}(\ud h(t)-x)\]
for $t>T$ and $0\le x\le \ud h(t)$. Now we verify that $(\ud u,\ud v,\ud h)$ satisfies
\bes\begin{cases}\label{3.6}
\ud u_t\le d_1\ud u_{xx}-a\ud u+H(\ud v), &t>T,~x\in(0,\ud h(t)),\\
 \ud v_t\le d_2\ud v_{xx}-b\ud v+G(\ud u), &t>T,~x\in(0,\ud h(t)),\\
\ud u(t,0)\le u(t,0), ~ \ud v(t,0)\le v(t,0),  &t>T,\\
\ud u(t,\ud h(t))\le0, ~ \ud v(t,\ud h(t))\le0, &t>T,\\
\ud h'(t)\le-\mu_1 \ud u_x(t,\ud h(t))-\mu_2\ud v_x(t,\ud h(t)), & t>T,\\
\ud h(T)\le h(T), ~ \ud u(T,x)\le u(T,x), ~ \ud v(T,x)\le v(T,x),&0\le x\le \ud h(T).
 \end{cases}
 \ees
Once it is done, we immediately derive
\[h(t)\ge\ud h(t), ~ u(t,x)\ge\ud u(t,x), ~ v(t,x)\ge\ud v(t,x) ~ ~ {\rm for ~ }t\ge T, ~ x \in[0,\ud h(t)),\]
which implies
 \bess\liminf_{t\to\yy}\frac{h(t)}{t}\ge \lim_{t\to\yy}\frac{\ud h(t)}{t}= c_{\mu_1,\mu_2,2\delta}.\eess
Together with $c_{\mu_1,\mu_2,2\delta}\to c_{\mu_1,\mu_2}$ as $\delta\to0$, the estimate \qq{3.8} is obtained.

Now we prove \eqref{3.6}. Simple calculations yield that for
$(t,x)\in[T,\yy)\times(0,\ud h(t))$,
 \bess
 &&\ud u_t-d_1\ud u_{xx}+a\ud u-H(\ud v)=-2\delta\ud u\le0,\\
 &&\ud v_t-d_2\ud v_{xx}+b\ud v-G(\ud u)=-2\delta\ud v\le0.
 \eess
So the first two inequalities of \eqref{3.6} holds. Due to the choices of $\delta$ and $T$, we see
\[(u(t,x),v(t,x))\ge(u^*_{\delta},v^*_{\delta})\ge (u^*_{2\delta},v^*_{2\delta})\ge(\ud u(t,x),\ud v(t,x))~ ~{\rm for ~ }t\ge T, ~ x\in[0,L],\]
which, combined with the definition of $(\ud u,\ud v,\ud h)$, leads to the inequalities in the third, fourth and sixth lines of \eqref{3.6}. Moreover, a straightforward computation gives
\[-\mu_1 \ud u_x(t,\ud h(t))-\mu_2\ud v_x(t,\ud h(t))=\mu_1\varphi'_{2\delta}(0)+\mu_2\psi'_{2\delta}(0)=c_{\mu_1,\mu_2,2\delta}=\ud h'(t).\]
Therefore, \eqref{3.6} holds.

To finish the proof, it remains to show the asymptotical behavior of $(u,v)$, i.e., \eqref{3.9}. For any $c\in[0,c_{\mu_1,\mu_2})$, we may choose $\delta$ small enough such that $c<c_{\mu_1,\mu_2,2\delta}$. From the definition of $(\ud u,\ud v,\ud h)$, it follows that $\max_{x\in[0,ct]}\big(|\ud u(t,x)-u^*_{2\delta}|+|\ud v(t,x)-v^*_{2\delta}|\big)\to0$ as $t\to\yy$, which implies $\liminf_{t\to\yy}(u(t,x),v(t,x))\ge(u^*_{2\delta},v^*_{2\delta})$ uniformly in $[0,ct]$. The arbitrariness of $\delta$ yields
 \[\liminf_{t\to\yy}(u(t,x),v(t,x))\ge(u^*,v^*) ~ ~{\rm uniformly ~ in ~ }[0,ct].\]
On the other hand, as in the proof of Lemma \ref{l3.1}, we see
\[\limsup_{t\to\yy}(u(t,x),v(t,x))\le(u^*,v^*) ~ ~{\rm uniformly ~ in ~ }[0,\yy).\]
The limit \qq{3.9} is derived. We complete the proof.
\end{proof}
Obviously, Theorem \ref{t1.2} follows from Lemmas \ref{l3.1}, \ref{l3.2}, \ref{l3.4}, \ref{l3.5} and \ref{l3.6}.

\end{document}